\ifdefined\submit 
\documentclass[review,hidelinks,onefignum,onetabnum]{siamart220329}
\else 
\documentclass[letter,11pt]{article}
\fi 

\ifdefined\submit 
\usepackage{amsfonts}
\usepackage{graphicx}
\usepackage{algorithmic}
 \usepackage{soul}
\newsiamremark{remark}{Remark}
\newsiamremark{hypothesis}{Hypothesis}
\crefname{hypothesis}{Hypothesis}{Hypotheses}
\newsiamthm{claim}{Claim}
\newsiamthm{assumption}{Assumption}

\headers{A four-operator splitting algorithm}{J.H. Alcantara, C.-p. Lee, A. Takeda}

\title{A four-operator splitting algorithm for nonconvex and nonsmooth optimization\thanks{%\funding{\textcolor{red}
{This work was funded by JSPS KAKENHI (23H03351,23K19981,24K20845).}}}%}

\author{Jan Harold Alcantara\thanks{Center for Advanced Intelligence Project, RIKEN
		(\email{janharold.alcantara@riken.jp}).}
	\and Ching-pei Lee\thanks{Department of Advanced Data Science, Institute of Statistical Mathematics, Tokyo, Japan 
		(\email{chingpei@ism.ac.jp})}
	\and Akiko Takeda\footnotemark[2] \thanks{Department of Mathematical Informatics, Graduate School of Information Science and Technology, The University of Tokyo, Tokyo, Japan (takeda@mist.i.u-tokyo.ac.jp)}}

\usepackage{amsopn}

\else 
\usepackage{amsthm}
 \usepackage[pdftex]{graphicx}
 \usepackage[framemethod=tikz]{mdframed}
 \usepackage[ruled,vlined]{algorithm2e}
 \usepackage{epsfig,epstopdf}
 \usepackage{mathrsfs}
 \usepackage{bbm}
 \usepackage{seqsplit}
 \usepackage{caption}   
 \usepackage{float} 
 \usepackage{bigstrut}
 \usepackage{authblk}
  \usepackage{tcolorbox}
 \usepackage[round]{natbib}
\usepackage[breaklinks=true,pdfstartview=FitH,pagebackref=true]{hyperref}
\usepackage{footnotebackref}

\ifdefined\submit

\else
\fi

 \usepackage{color}
\usepackage{subcaption}

\fi

\usepackage{mathtools}
\usepackage{amssymb}
\usepackage[shortlabels]{enumitem}
\usepackage[T1]{fontenc}
\usepackage{multirow}
\usepackage{booktabs}
\usepackage{wrapfig}

 \newcommand{\ds}{\displaystyle}

\newcommand{\Lh}{L_{_{\rm{h}}}}

\newcommand{\Lf}{L_{_{\rm{f}}}}
\newcommand{\rhof}{\rho_{_{\rm{f}}}}
\newcommand{\rhoh}{\rho_{_{\rm{h}}}}
\newcommand{\rhop}{\rho_{_{\rm{p}}}}
\newcommand{\sigmaf}{\sigma_{_{\rm{f}}}} 
\newcommand{\sigmaphi}{\sigma_{\phi}}
\newcommand{\sigmah}{\sigma_{_{\rm{h}}}}
\newcommand{\rhophi}{L_{\phi}}
\newcommand{\bftab}{\fontseries{b}\selectfont}
\newcommand{\plus}{{\scriptscriptstyle +}}

\newcommand{\cpsolved}[1]{}

\newcommand{\jhsolved}[1]{}
\newcommand{\memosolved}[2]{}
\ifdefined\submit
\newcommand{\comment}[1]{}
\else
\renewcommand{\comment}[1]{}
\fi 

\newcommand{\conditionalstar}[1]{%
  \ifdefined\submit
    *
  \else
    #1%
  \fi
}
 
  \newcommand{\cpmodifyok}[1]{#1}

\ifdefined\submit 
 \newcommand{\alert}[1]{\textcolor{magenta}{#1}}
 \else
  \newcommand{\alert}[1]{#1}
  \fi

\def \rla{\right\rangle}
\def \lla{\left\langle}
\def\citep{\cite}
\def\citet{\cite}

\renewcommand{\Re}{{\rm I}\! {\rm R}}

\DeclareMathOperator*{\argmin}{arg\,min}

\DeclareMathOperator*{\Fix}{Fix}

\DeclareMathOperator*{\dist}{dist}

\DeclareMathOperator*{\dom}{dom}

\newcommand{\norm}[1]{{\left\|{#1}\right\|}}
\newcommand{\prox}{{\rm prox}}

\ifdefined\submit 
\else 
\usepackage[margin=1in]{geometry}
\usepackage[capitalize,nameinlink]{cleveref}
\crefname{assumption}{Assumption}{Assumptions}
\crefformat{equation}{\textup{#2(#1)#3}}
\crefrangeformat{equation}{\textup{#3(#1)#4--#5(#2)#6}}
\crefmultiformat{equation}{\textup{#2(#1)#3}}{ and \textup{#2(#1)#3}}
{, \textup{#2(#1)#3}}{, and \textup{#2(#1)#3}}
\crefrangemultiformat{equation}{\textup{#3(#1)#4--#5(#2)#6}}%
{ and \textup{#3(#1)#4--#5(#2)#6}}{, \textup{#3(#1)#4--#5(#2)#6}}{, and \textup{#3(#1)#4--#5(#2)#6}}

\Crefformat{equation}{#2Equation~\textup{(#1)}#3}
\Crefrangeformat{equation}{Equations~\textup{#3(#1)#4--#5(#2)#6}}
\Crefmultiformat{equation}{Equations~\textup{#2(#1)#3}}{ and \textup{#2(#1)#3}}
{, \textup{#2(#1)#3}}{, and \textup{#2(#1)#3}}
\Crefrangemultiformat{equation}{Equations~\textup{#3(#1)#4--#5(#2)#6}}%
{ and \textup{#3(#1)#4--#5(#2)#6}}{, \textup{#3(#1)#4--#5(#2)#6}}{, and \textup{#3(#1)#4--#5(#2)#6}}

\crefdefaultlabelformat{#2\textup{#1}#3}

 \numberwithin{equation}{section}
\newtheorem{theorem}{Theorem}[section]

\newtheorem{proposition}[theorem]{Proposition}
\newtheorem{lemma}[theorem]{Lemma}
\newtheorem{assumption}[theorem]{Assumption}
\theoremstyle{definition} 	

\newtheorem{remark}[theorem]{Remark}

\def\TheTitle{A four-operator splitting algorithm for nonconvex and nonsmooth optimization}

\ifpdf
\hypersetup{
  pdftitle={\TheTitle},
  pdfauthor={Jan Harold Alcantara, Ching-pei Lee and Akiko Takeda}}
\fi

 \title{\TheTitle}
 \date{\today}
\author{Jan Harold Alcantara\thanks{\url{janharold.alcantara@riken.jp}.  Center for Advanced Intelligence Project, RIKEN, Tokyo, Japan.} \qquad Ching-pei Lee\thanks{\url{chingpei@ism.ac.jp}.  Department of Advanced Data Science, Institute of Statistical Mathematics, Tokyo, Japan.}
\qquad Akiko Takeda\thanks{\url{takeda@mist.i.u-tokyo.ac.jp}.
Department of Mathematical Informatics, Graduate School of Information Science and Technology, University of Tokyo,
Tokyo, Japan, and Center for Advanced Intelligence Project, RIKEN, Tokyo, Japan}
}

\fi 

\begin{document}

\maketitle

 \begin{abstract}
In this work, we address a class of nonconvex nonsmooth optimization
problems where the objective function is the sum of two smooth
functions (one of which is proximable) and two nonsmooth functions
% (one weakly convex and proximable and the other continuous and weakly concave). We
(one proper, closed and proximable, and the other continuous and weakly concave). We
introduce a new splitting algorithm that extends the Davis-Yin
splitting (DYS) algorithm to handle such four-term nonconvex nonsmooth
problems.  We prove that with appropriately chosen stepsizes, 
our algorithm exhibits global subsequential convergence to stationary
points with a stationarity measure converging at a global
rate of $1/T$, where $T$ is the number of iterations. When
specialized to the setting of the DYS algorithm, our results allow for
larger stepsizes compared to existing bounds in the literature.
Experimental results demonstrate the practical applicability and
effectiveness of our proposed algorithm. \\

\ifdefined\submit
\else 
\fi 
   
\end{abstract}

\ifdefined\submit 
% REQUIRED

% REQUIRED
\else 
\fi 

\section{Introduction}
We consider the nonsmooth and nonconvex problem
\begin{equation}
	\min_{x\in\Re^n} ~\Psi (x) \coloneqq f(x) + g(x) + h(x) + p(x)
	\label{eq:problem}
\end{equation}
under the following assumptions.
\begin{assumption} 
	\label{assume:problem}
	The functions, $f$, $g$, $h$ and $p$ satisfy the following:\footnote{\cpmodifyok{See
	\cref{sec:prelim} for the definitions of $L$-smooth,
$\rho_g$-weakly convex, and proximal mapping.}}
	\begin{enumerate}[(a)]
		\item $f:\Re^n\to(-\infty,\infty)$  is an $\Lf$-smooth function and is ``proximable'', in the sense that its proximal mapping either has a closed form or is easily computable;  
		% \item $g:\Re^n\to (-\infty,\infty]$ is a proper, closed, and $\rhog$-weakly convex function; that is, $g+\frac{\rhog}{2}\|\cdot\|^2$ is convex where $\rhog\geq 0$;
        \alert{\item $g:\Re^n\to (-\infty,\infty]$ is a proper, closed, and proximable function;}
        \item $h:\Re^n\to (-\infty,\infty)$ is an $\Lh$-smooth function; and 
		\item $p:\Re^n\to \Re$ is continuous (possibly nonsmooth) on an open set containing the domain of $g$ such that 
  $-p$ is \alert{$\rhop$}-weakly convex with $\alert{\rhop}\geq
  0$\cpsolved{\st{; that is, $\frac{\alert{\rhop}}{2}\|\cdot\|^2 - p$
  is convex}}.
	\end{enumerate}
\end{assumption}

A notable special case of  \eqref{eq:problem} is the three-term
optimization problem, where $p \equiv 0$. The Davis-Yin splitting (DYS)
algorithm, introduced in \citep{DavisYin17}, tackles this problem for
convex functions $f$, $g$, and $h$, assuming $h$ additionally meets
\cref{assume:problem}(c). To extend the algorithm's applicability to
nonconvex problems, \cite{BianZhang21} leveraged techniques from
\citep{li2016douglas} to show that with a sufficiently small stepsize, the DYS algorithm achieves global subsequential convergence.
% This requires $f$ and $h$ to satisfy conditions (a) and (c) of
% \cref{assume:problem}, respectively, 
% and $g$ to be any proper closed function. 
This requires $f$, $g$ and $h$ to satisfy conditions (a), (b) and (c) of \cref{assume:problem}, respectively.
However, despite not requiring the functions to be convex,
this extension of the DYS algorithm is still limited as it can only admit at most one nonsmooth function.

In scenarios involving objective functions comprised of multiple
nonsmooth, nonconvex terms $-$ such as those encountered in nonconvex
fused Lasso \citep{parekh2015convex}, $\ell_{1-2}$  regularized
optimization problems \citep{yin2015minimization}, simultaneous sparse
and low-rank optimization problems \citep{richard2012estimation},
among others $-$ the DYS algorithm becomes inapplicable. Indeed,
the problem setting of the DYS framework is generally not suited for
several classes of nonsmooth nonconvex optimization problems such as
the following ones.
\ifdefined\submit 
\paragraph{Nonsmooth nonconvex optimization with concave component}
\else 
\paragraph{Nonsmooth nonconvex optimization with concave component.}
\fi
The DYS algorithm is not applicable to problems that  take the form
\begin{equation}
    \min _{x\in \Re^n} ~F(x) + G(x) - H(x),
    \label{eq:dc}
\end{equation}
where $F$ is $L$-smooth, $G$ is proper and closed, and $H$ is convex. The limitation of the DYS algorithm arises from the
potential nonsmoothness of the concave component $-H$. Meanwhile, by setting $f=F$, $g=G$, $h\equiv 0$ and $p\coloneqq-H$ in \eqref{eq:problem}, we satisfy \cref{assume:problem} with  $\Lh = \rhop = 0$.
\ifdefined\submit
\paragraph{General nonsmooth nonconvex regularized optimization}
\else 
\paragraph{General nonsmooth nonconvex regularized optimization.}
\fi 
Another example is the general class of regularized problems considered in \citep{LiuPongTakeda19}, given by 
\[ \min_{x\in \Re^n}~h(x) + g(x) + \sum_{i=1}^r P_i(A_ix),\]
where $g$ and $P_i$ are nonsmooth, nonconvex and nonnegative
regularizers, and  $A_i\in \Re^{n_i\times n}$. This difficult class of
optimization problems was addressed by \citet{LiuPongTakeda19} through the use of the Moreau envelope $M_{\lambda_i P_i}$ of each $P_i$ to obtain an approximate problem: 
    \begin{equation}
	\min_{x\in \Re^n}~h(x) + g(x) + \sum_{i=1}^r M_{\lambda_i P_i}(A_ix).
			\label{eq:moreau_approximate}
		\end{equation}
However, note that even after this reformulation, the DYS algorithm is
still not applicable to \cref{eq:moreau_approximate}, as the third
term is in general a nonsmooth function due to the nonconvexity of the
$P_i$'s. On the other hand, we can utilize the fact (see \cite[Proposition
3]{Lucet2006}) that for any function $P$, 
		\begin{equation}
		 \frac{1}{2\lambda}\|x\|^2  - M_{\lambda P} ( x) =  \sup_{y\in \dom (P)} \left\lbrace \frac{1}{\lambda
		}x^\top y- \frac{1}{2\lambda}\|y\|^2-P(y) \right\rbrace.
		\label{eq:dc_moreau}
		\end{equation}
By noting the convexity of the function \memosolved{AT}{Write "the optimal value function with respect to $x$"? Or it is obvious from this context? Not necessary to mention something about $P$, like "any nonnegative proper closed function" or something?}\jhsolved{I think the ``optimal value'' part is clear. And checking again \cite[Proposition 3]{Lucet2006}, there's no requirement for $P$ to be nonnegative. It can be an arbitrary function. } on the right-hand side of
\eqref{eq:dc_moreau} (being the supremum of affine (convex) functions)\jhsolved{If a justification as to why the RHS is convex is needed, how about we say ``being the supremum of affine (convex) functions'' as a more fundamental reason?}\memosolved{AT}{this reason: being the supremum of affine (convex) functions is easy to understand for me}\cpsolved{OK, I am also ok with that.},
	we may then fit \eqref{eq:moreau_approximate} in the form \eqref{eq:problem} by setting $f\equiv 0$ and $p \coloneqq \sum_{i=1}^r M_{\lambda_i P_i}\circ A_i$ so that \cref{assume:problem}(d) holds with $\rhop = \sum_{i=1}^r \lambda_i^{-1}$.
\ifdefined\submit 
% A special class of problems that conform with \eqref{eq:dc_moreau} includes nonconvex feasibility problems reformulated as optimization problems, as described in \cite{AlcantaraLeeTakeda2024}.
\else 
\paragraph{Nonconvex feasibility problems.} We can also draw examples from feasibility problems reformulated as optimization problems: Let $A,B,C$ and $D_i$ for $i\in \{ 1,2,\dots,r\}$ be closed nonempty sets where $A$ and $C$ are convex, and the sets $B$ and $D_i$ $(i\in \{1,\dots,r\})$ are not necessarily convex. The \textit{feasibility problem} is given by:
	\begin{equation*}
	\text{Find}~x\in A\cap B\cap C \cap \bigcap_{i=1}^r D_i.
	\label{eq:feasibility}
	\end{equation*}
By defining $d_S(x) \coloneqq \inf_{w\in S} \norm{w-x}$ for any
	set $S$, this problem can be reformulated as the following problem that
 conforms the form of \eqref{eq:problem}:
	\begin{equation}
	    \min_{x\in \Re^n}~ \underbrace{\frac{1}{2}d^2_A(x)}_{f(x)} + \underbrace{\delta_B(x)}_{g(x)} + \underbrace{\frac{1}{2}d^2_C(x)}_{h(x)}  + \underbrace{\sum_{i=1}^r \frac{1}{2}d^2_{D_i}(x)}_{p(x)}.
     \label{eq:feasibility_optimization}
	\end{equation}
The original setting of the DYS algorithm in \citep{DavisYin17} can
handle this problem when $r=0$. On the other hand, the extended
setting considered in \citep{BianZhang21} can handle the case
$B=\Re^n$ and $r=1$ (\textit{i.e.}, the feasibility problem with two convex
sets and one nonconvex set), but it requires a very small stepsize.
Meanwhile, by the convexity of $A$, $B$ and $C$, we see that
\cref{assume:problem}(a), (b) and (c) hold for
\cref{eq:feasibility_optimization}. The function $p$, on the other hand, is a nonsmooth function due to the nonconvexity of each $D_i$. Nevertheless, noting that the
	half-squared distance function is the Moreau envelope of the
	indicator function, we have from \eqref{eq:dc_moreau} that $p$
	satisfies \cref{assume:problem}(d) with $\rhop = r$.
	\memosolved{AT}{$r$ is the number of summations for the last term? $\rightarrow$ Got it. Thanks}\cpsolved{Each distance
	function is $1$-smooth, so indeed it is what you said.}\jhsolved{Actually, here, each function $\frac{1}{2}d^2_{D_i}(\cdot)$ is not a smooth function, but it satisfies the condition in \cref{assume:problem}(d): we have $\frac{1}{2}\norm{x}^2 - \frac{1}{2}d^2_{D_i}(x) = \frac{1}{2}\norm{x}^2 - M_{\delta_{D_i}}(x)$, which is convex by \eqref{eq:dc_moreau}.}
 \fi 
\medskip 

\alert{We note that the optimization problem \eqref{eq:problem} under \cref{assume:problem} can be reformulated into the form \eqref{eq:dc} by setting $H \coloneqq p -\frac{\rhop}{2}\norm{\cdot}^2 $ and choosing either 
(i) $F \coloneqq f+h+\frac{\rhop}{2}\norm{\cdot}^2 $, $G\coloneqq g$, or (ii) $F \coloneqq f+h $, $G\coloneqq g+\frac{\rhop}{2}\norm{\cdot}^2$. Under such reformulations, one may apply existing algorithms such as the generalized proximal point (GPP) method \cite{An2016}, which coincides with the proximal difference-of-convex (PDC) algorithm \cite{WenChenPong2018} when $F$ and
$G$ are convex.}

\alert{However, one of the advantages of the framework we propose is
	that it enables a direct treatment of the problem
	\eqref{eq:problem} without requiring such reformulations, and it
	takes advantage of the decomposition of the smooth term of the
	objective into $f+h$,  where $f$ and $h$ possess the structures
	specified in \cref{assume:problem}(a) and (c),
	respectively.\footnote{\alert{In practice, when both terms satisfy
		\cref{assume:problem}(a), it is often more practical to assign
	more structure to $f$, as the proximal step tends to exploit more
information than the gradient step.}}
\cpsolved{It looks to me that these 2 paragraphs should be in
not the contribution section but the previous section. For
contributions we should be more succinct and to the point.
It can be combined with a guideline for choosing $f$ and $h$
(basically taking as much to $f$ as possible as prox gouges more
information than grad, I think) suggested by Akiko.}}\jhsolved{I now moved the 2 paragraphs you mentioned to this section. I also added a sentence but as a footnote only. What do you think?}
\subsection{Contributions}
\alert{In this paper, we propose an extension of the Davis–Yin splitting (DYS) algorithm to handle optimization problems involving two nonsmooth terms, thereby expanding its application to a broader class of nonconvex and nonsmooth problems, including those described earlier. Specifically, we develop a four-operator splitting algorithm for the general problem \eqref{eq:problem} under the setting described in \cref{assume:problem}. Our main contributions are summarized below:
\begin{enumerate}[(I)]
    \item We prove that our algorithm, which includes the DYS
		splitting as a special case, is globally subsequentially
		convergent to stationary points under appropriately chosen
		stepsizes. Furthermore, a first-order optimality measure is
		shown to converge at a \cpmodifyok{global} rate of
		$1/\cpmodifyok{T}$, \cpmodifyok{where $T$ is the number of
		iterations}.
    % \item Another major contribution of this work is the derivation of upper bounds for stepsizes that guarantee these convergence results. When specialized to $p\equiv0$, our results yield stepsize estimates for the DYS algorithm that are significantly larger than the existing bounds in \citep{BianZhang21}. As a practical consequence, our experiments demonstrate that the derived larger stepsizes for the DYS algorithm result in faster convergence of the algorithm.
    \item A key theoretical contribution of this work is the
		derivation of upper bounds on the \cpmodifyok{stepsizes}\cpsolved{For
			consistency, in other places we all used stepsize instead
		of step size.}
		that guarantee convergence. When specialized to the case of $p
		\equiv 0 $, our results yield \cpmodifyok{stepsize} estimates
		for the DYS algorithm that are significantly larger than
		existing bounds in \citep{BianZhang21}. As a practical
		consequence, our experiments show that these larger \cpmodifyok{stepsizes} lead to faster convergence of the algorithm.
    \item[(III)]  Our experiments demonstrate that leveraging the structures of $f$ and $h$ separately leads to faster convergence compared to GPP/PDC, which treats the smooth term as a single component. Hence, while reformulation into \eqref{eq:dc} is possible, our direct approach offers practical advantages due to utilization of individual structures. 
\end{enumerate}}
% Another major contribution of this work is the derivation of upper
% bounds for stepsizes that guarantee these convergence results. When specialized to $p\equiv0$, our results yield stepsize estimates for the DYS algorithm that are significantly larger than the existing bounds in \citep{BianZhang21}.
% To our knowledge, for this special case of $p\equiv0$, our iteration
% complexity result is also new for the DYS algorithm.

\alert{During the final stage of preparing the first draft of this
	paper, we came across the preprint \cite{Dao2024}, which
	introduces the ``doubly relaxed forward-Douglas–Rachford (DRFDR)''
	splitting algorithm. This method applies to a special case of our
	problem setting where $\rhop=0$ in \cref{assume:problem}(d). When
	$\rhop>0$, we note that DRFDR can also be applied to
	\eqref{eq:problem} by subtracting $\frac{\rhop}{2}\norm{\cdot}^2$
	from $p$, and then adding $\frac{\rhop}{2}\norm{\cdot}^2$ to
	either $f$, $g$ or $h$. However, this transformation introduces an
	additional step, and determining where it is most practical to add this
	quadratic term is not straightforward. In
	contrast, our proposed algorithm accommodates the case $\rhop>0$
	without requiring such transformations. More importantly, our
	convergence analysis relies on entirely different proof
	techniques, enabling us to derive larger stepsize bounds compared
	to those in \cite{Dao2024}. Our numerical experiments demonstrate that these larger stepsizes lead to improved practical performance, highlighting another key distinction between our work and \cite{Dao2024}.}

\subsection{Organization} This paper is organized as follows. In \cref{sec:prelim}, we summarize
important definitions, notations and known results that we will use in this
paper. We propose our algorithm and establish its global subsequential
convergence and convergence rates in \cref{sec:algorithms}. Experiments to demonstrate the applicability and efficiency of our method are presented in \cref{sec:experiments}, and concluding remarks are given in \cref{sec:conclusion}. 

\section{Preliminaries}\label{sec:prelim}
Throughout the paper, $\Re^n$ denotes the $n$-dimensional Eu\-clid\-ean space endowed with the
inner product $\lla \cdot, \cdot \rla$, and we denote its induced
norm by $\|\cdot \|$. The set $(-\infty,\infty]$ is the extended
real-line, and we adopt the conventions that $\frac{a}{\infty}=0$, $\frac{a}{0}=\infty$ for any $a\neq 0$, and $\frac{\infty}{\infty}=1$. 

Let  $\phi:\Re^n\to (-\infty,\infty]$ be an extended-valued function. The \emph{domain} of $\phi$ is given by the set $\dom (\phi) = \{ x\in\Re^n : \phi(x)<\infty\}$. We say that $\phi$ is a \emph{proper} function if $\dom (\phi)\neq \emptyset$, and that $\phi$ is \emph{closed} if it is lower semicontinuous. 
%
%\medskip 
$\phi$ is said to be a \emph{$\sigmaphi$-convex} function if
$\phi-\frac{\sigmaphi}{2}\norm{\cdot}^2$ is convex for some
$\sigmaphi\in\Re$.  If $\sigmaphi>0$, then $\phi$ is
\emph{$\sigmaphi$-strongly convex}. If $\sigmaphi\leq 0 $, we denote
$\rho_{\phi}\coloneqq -\sigmaphi$ and call $\phi$ a
\emph{$\rho_{\phi}$-weakly convex function}. In other words, $\phi$ is
$\rho_{\phi}$-weakly convex for $\rho_{\phi} \geq 0$ if $\phi+\frac{\rho_{\phi}}{2}\norm{\cdot}^2$ is convex.

\medskip 
The \emph{subdifferential} of $\phi$ at a point $x\in \dom (\phi)$  is defined as
\begin{equation}
    \partial \phi (x) \coloneqq \left\lbrace \xi \in \Re^n : \exists
	\{(x^k,\xi^k)\} \text{ such that }
	x^k\xrightarrow{\phi} x,  ~\xi^k\in \hat{\partial}\phi(x^k),
	~\text{and}~\xi^k\to\xi\right\rbrace,
	\label{eq:generalsubdiff}
\end{equation}
where $x^k\xrightarrow{\phi} x$ means $x^k\to x$ and $\phi(x^k)\to \phi(x)$, and
\begin{equation*}
    \hat{\partial}\phi (x) \coloneqq \left\lbrace \xi\in\Re^n:
	\liminf_{\bar{x}\to x, \bar{x}\neq x}\frac{\phi(\bar{x}) - \phi(x)
	- \lla \xi,\bar{x}-x\rla }{\norm{\bar{x}-x}} \geq 0\right\rbrace.
\end{equation*}
When $\phi$ is convex, \cref{eq:generalsubdiff} coincides with the classical subdifferential in convex analysis:
\begin{equation*}
\partial \phi(x) = \{\xi\in \Re^n : \phi(y) \geq \phi(x) + 
\lla \xi, y-x \rla,
~\forall y\in \Re^n \}.
\label{eq:partial_phi_convex}
\end{equation*} If $\phi$ is continuously differentiable, the subdifferential reduces to a singleton containing the gradient of $\phi$. We also note that the definition of the subdifferential gives the following property:
\begin{equation}
\left\lbrace \xi\in\Re^n : \exists \{ (x^k,\xi^k) ~\text{such that} ~x^k\xrightarrow{\phi} x,~\xi^k\in \partial \phi (x^k),~\text{and}~\xi^k \to \xi \right\rbrace \subseteq \partial \phi (x).
\label{eq:generalsubdiff_consequence}
\end{equation}
For a proper and closed function $\phi$, we define the \emph{proximal mapping} of $\phi$ as
\begin{equation}
    \prox_{\gamma \phi} (x) \coloneqq \argmin_{w\in \Re^n}\, \phi(w) + \frac{1}{2\gamma}\|w-x\|^2, \quad \gamma>0.
	\label{eq:prox} 
\end{equation}
For a set $S\subseteq \Re^n$, we define $\prox_{\gamma \phi} (S)
\coloneqq \bigcup_{x\in S}\prox_{\gamma \phi}(x)$. From the optimality condition of \eqref{eq:prox}, we have 
\begin{equation}
y\in \prox_{\gamma \phi}(x) \quad \Longrightarrow \quad x-y \in \gamma \partial \phi (y),
\label{eq:prox_optimality}
\end{equation}
and the converse holds if $\phi + \frac{1}{2\gamma}\norm{\cdot}^2$ is convex.
We also use the notation $T_{\gamma \phi}: \Re^n
\rightrightarrows \Re^n$ to denote %the \textit{forward operator}
%associated with $\phi:\Re^n\to \Re$ with stepsize $\gamma>0$, which is
%given by
\ifdefined\submit 
$T_{\gamma \phi}\coloneqq Id - \gamma \partial \phi,$
\else
\begin{equation*}
	T_{\gamma \phi}\coloneqq Id - \gamma \partial \phi,
\end{equation*} 
\fi 
where $Id$ is the identity map on $\Re^n$.
Given a
point $x$ and a subgradient $\xi\in \partial
\phi(x)$,\cpsolved{``generalized'' is not defined. Given the descriptions in
the previous section, should we just call it subgradient?} for any
$\gamma > 0$, we define %the quadratic approximation of $\phi$ at $x$ as
\[ Q_ {\gamma \phi}(w; x,\xi) \coloneqq \phi(x) + \lla \xi , w-x \rla
+ \frac{1}{2\gamma}\norm{w-x}^2,\quad \forall w\in \Re^n. \]
When $\phi$ is smooth, we write $Q_ {\gamma \phi}(w; x)$, and $\xi$ is understood to be equal to $\nabla \phi(x)$.

Let $\phi:\Re^n\to\Re$ be a continuously differentiable function $\phi:\Re^n\to\Re$. We say that $\phi$ is \emph{$L_{\phi}$-smooth} if its gradient satisfies 
\ifdefined\submit 
$\norm{\nabla \phi (x) - \nabla \phi (y)}\leq L_{\phi}\norm{x-y}$ for all $ x,y\in \Re^n.$
\else 
\[ \norm{\nabla \phi (x) - \nabla \phi (y)}\leq L_{\phi}\norm{x-y}, \quad \forall x,y\in \Re^n.\]
\fi 
We collect some properties of $L_{\phi}$-smooth functions. 

\begin{lemma}
    Let $\phi:\Re^n\to\Re$ be an $L_{\phi}$-smooth function. Then for any $x,y\in \Re^n$, 
    \begin{enumerate}[(a)]
        \item  $ \left| \phi(y)-\phi(x)-\lla \nabla \phi(x) , y-x \rla
			\right|  \leq \frac{L_{\phi}}{2}\norm{y-x}^2$. \quad (Descent Lemma)
        \item If in addition, $\phi$  is $\rho_\phi$-weakly convex, then for
			any $L\geq \rhophi$ such that $L > \rho_\phi$, we have \ifdefined\submit 
            $\phi(y)-\phi(x)-\lla \nabla \phi(x) , y-x \rla  \geq
				\frac{1}{2(L-\rho_{\phi})}\norm{\nabla \phi(y)-\nabla
				\phi(x)}^2 -
				\frac{\rho_{\phi}L}{2(L-\rho_{\phi})}\norm{y-x}^2.$
                \else 
			\[
				\phi(y)-\phi(x)-\lla \nabla \phi(x) , y-x \rla  \geq
				\frac{1}{2(L-\rho_{\phi})}\norm{\nabla \phi(y)-\nabla
				\phi(x)}^2 -
				\frac{\rho_{\phi}L}{2(L-\rho_{\phi})}\norm{y-x}^2.
			\]
                \fi 
    \end{enumerate}
    \label{lemma:properties_Lsmooth}
\end{lemma}
\begin{proof}
    Parts (a) and (b) follow from \cite[Proposition A.24]{Bertsekas2016}  and \cite[Theorem 2.2]{ThemelisPatrinos2018}, respectively. 
\end{proof}

\begin{remark}
    By \cref{lemma:properties_Lsmooth}(a), $\phi$ is $\rho_{\phi}$-weakly convex for some $\rho_{\phi}\in[0,L_{\phi}]$ if $\phi$ is $L_{\phi}$-smooth. If $\phi$ is convex and $L_{\phi}$-smooth, we obtain from \cref{lemma:properties_Lsmooth}(b) that %a familiar property that 
    \begin{equation}
    \phi(y)-\phi(x)-\lla \nabla \phi(x) , y-x \rla  \geq \frac{1}{2\rhophi}\norm{\nabla \phi(y)-\nabla \phi(x)}^2, \quad \forall x,y\in \Re^n.
    \label{eq:lowerbound_intermsofgrad}
    \end{equation}
    \label{remark:Lsmooth_implies_weaklyconvex}
\end{remark}

\section{Proposed algorithm and its convergence}
\label{sec:algorithms}
In this section, we propose our algorithm and provide its
convergence analysis under various parameter settings.

\subsection{Proposed algorithm}
\label{subsec:alg}
We present our proposed algorithm for solving \eqref{eq:problem} in \cref{alg:main}.

\begin{algorithm}%[tb]
    \ifdefined\submit
    \else
	\SetAlgoLined
	\fi
	\begin{description}
		\item[Step 0.] Choose an initial point $(y^0,z^0) \in
			\Re^n\times \Re^n$ and stepsize parameters $\tau > 0$, $\gamma\in (0,\infty)$, and $\alpha, \beta \in (0,\infty]$ such that $\frac{1}{\gamma} = \frac{1}{\alpha}+\frac{1}{\beta}$.
		\item[Step 1.] Set 
		\begin{align}
			x^{k}&  \in \prox_{\alpha f} (z^k), \label{eq:x-step} \\
			y^{k+1}&  \in \prox_{\gamma g} \left(
			\frac{\gamma}{\alpha} (2x^{k}-z^k-\alpha \nabla h(x^k))  +
			\frac{\gamma}{\beta} T_{\beta p} (y^k)  \right),  \label{eq:y-step} \\
			z^{k+1} & = z^k + \tau ( y^{k+1} - x^{k} ).\label{eq:z-step}
		\end{align}
		\item[Step 2.] If a termination criterion is not met, go to Step 1. 
	\end{description}
	\caption{A four-operator splitting algorithm for nonsmooth nonconvex optimization}
	\label{alg:main}
\end{algorithm}

For convenience in later discussions, we define $P_{\Lambda}:\Re^n\times \Re^n \rightrightarrows \Re^n$ as 
\memosolved{AT}{Does $P_{\Lambda} (y,z)$ change depending on which $x$ is
chosen?}\cpsolved{I think that is a set operator, though the notation is
	kind of ambiguous, but it should be the set of the union over all
	$x$ (and all elements of $T_{\beta p}(y)$.}
 \jhsolved{That's right, it should be the union over all $x$. Actually, this wouldn't matter later on as we would require $\alpha$ to be strictly less than $1/(\Lf+\Lh)$, and in turn, $\alpha<1/\Lf$. In this case, $\prox_{\alpha f}$ would be single-valued anyway... But this is for later when we talk about the proper stepsize to guarantee strict decrease. So I will make the necessary modifications for the first parts where we don't require yet the restriction on $\alpha$.}
	\begin{equation}
		P_{\Lambda} (y,z) \coloneqq \bigcup_{x\in \prox_{\alpha f}(z)}\prox_{\gamma g} \left( \frac{\gamma}{\alpha} (2x-z-\alpha \nabla h(x))+   \frac{\gamma}{\beta} T_{\beta p} (y)  \right), 
		\label{eq:Plambda}
	\end{equation}
	where $\Lambda \coloneqq (\alpha,\beta,\gamma)$, $\alpha,\beta\in (0,\infty]$ and $\gamma \in (0,\infty)$. Hence, the $y$-step in \eqref{eq:y-step} can be written as $y^{k+1} \in P_{\Lambda}(y^k,z^k)$. Observe that since $f$ is differentiable, $x
\in \prox_{\alpha f}(z)$ implies that $z-x  = \alpha \nabla f(x)$ by \eqref{eq:prox_optimality}, so we may also write $P_{\Lambda}$ as 
	\begin{equation}
P_{\Lambda} (y,z) = \bigcup_{x\in \prox_{\alpha f}(z)} \prox_{\gamma g} \left( \frac{\gamma}{\alpha} T_{\alpha (f+h)}(x)  +   \frac{\gamma}{\beta} T_{\beta p} (y)  \right).
\label{eq:Plambda2}
\end{equation}
%
%, and $\gamma \coloneqq (\alpha^{-1}+\gammah^{-1}+  \beta^{-1})^{-1}$.
When $\rhop=0$, we in particular set $\beta=\infty$, so that $\gamma = \alpha$ by Step 0 of \cref{alg:main}. On the other hand, when $\Lf=\Lh= 0$, we set $\alpha=\infty$ and therefore $\gamma=\beta$.

\begin{remark}
\label{rem:alg_specialcases}
The above algorithm covers several algorithms in the literature.
\begin{enumerate}
	\item (Davis-Yin splitting algorithm). When $p\equiv 0$, $\tau=1$,
		and $\gamma=\alpha$, the algorithm reduces to the Davis-Yin
		splitting (DYS) algorithm \citep{DavisYin17}, which itself
		covers the gradient descent algorithm (when $f\equiv g \equiv
		0$), the proximal gradient algorithm for the sum of a smooth
		and a nonsmooth function (when $f\equiv 0$), and the
		Douglas-Rachford algorithm
		\citep{DouglasRachford56a} (when $h\equiv 0$). 
	
	\item (Proximal subgradient method for sum of two nonsmooth functions). When $f\equiv h\equiv 0$ so that the stepsizes satisfy $\gamma = \beta$ and $\alpha= \infty$, the algorithm reduces to the proximal subgradient algorithm: 
 \ifdefined\submit
 $ y^{k+1} = \prox_{\gamma g }(y^k-\gamma \partial p(y^k),$
 \else 
 \[ y^{k+1} = \prox_{\gamma g }(y^k-\gamma \partial p(y^k)),\]
 \fi 
    but covers a wider range of problems compared to DYS with $f\equiv 0$ since the function $p$ could be nonsmooth. 
	
	\item (Generalized proximal point algorithm/Proximal DC algorithm). When $f\equiv 0$, the algorithm simplifies to 
    \begin{align*}
        y^{k+1} \in \prox_{\gamma g} \left( \frac{\gamma}{\alpha} (z^k-\alpha \nabla h(z^k))  +   \frac{\gamma}{\beta} T_{\beta p} (y^k)  \right), \,
        z^{k+1} = (1-\tau )z^k + \tau y^{k+1}.
    \end{align*}
    Thus, when $\tau=1$ and $\rhop=0$ so that  $\beta=\infty$ and $\gamma=\alpha$, we further obtain
	\begin{equation}
		y^{k+1} \in \prox_{\gamma g}(y^{k}-\gamma \nabla h(y^k)-\gamma \partial p (y^k)).
		\label{eq:proxdc}
	\end{equation}

When $g$ is convex, this method is known as the proximal DC (PDC) algorithm \citep{WenChenPong2018}, while for a general (not necessarily convex) $g$, it is referred to as the generalized proximal point (GPP) algorithm in \cite{An2016}.

\end{enumerate}	
\end{remark}

\subsection{Stepsize bounds}
We now derive appropriate stepsizes for \cref{alg:main} that will
guarantee sufficient descent of \alert{a certain merit function}. First, we
introduce some notations. Given $\xi \in \partial p(y)$, we define $ V_{\Lambda,\xi} :
\Re^n\times \Re^n \times  \Re^n \to \Re$ as
\begin{equation}
	 V_{\Lambda,\xi} (y,z,x) \coloneqq \min_{w\in \Re^n} ~\Phi_{\Lambda,\xi}(w;y,z,x),
	 \label{eq:envelope}
\end{equation}
where 
\begin{equation}
	\Phi_{\Lambda,\xi}(w;y,z,x) \coloneqq	Q_{\alpha (f+h)} ( w;x) +   Q_{\beta p} ( w;y,\xi) + g(w)  .
	\label{eq:Phi}
\end{equation}

We will use \eqref{eq:envelope} as a merit function for our proposed
algorithm. It is not difficult to verify that when $p\equiv 0$ and
$\alpha=\gamma$ as in the setting described in the first item of \cref{rem:alg_specialcases}, the function given by \eqref{eq:envelope} with $\alpha<\Lf^{-1}$ and $x=\prox_{\alpha f}(z)$ simplifies to the Davis-Yin envelope function introduced in \citep{LiuYin19}, which covers the forward-backward envelope \citep{ThemelisStellaPatrinos2018} and the Douglas-Rachford envelope \citep{PatrinosStellaBemporad2014}.

\alert{
\begin{lemma}
	\label{lemma:prox_formula}
	Suppose that $f$ and $h$ are continuously differentiable
	functions, 
 % $g$ satisfies \cref{assume:problem} (b), 
$g$ is proper and closed,  
 and $\gamma>0$ satisfies $\frac{1}{\gamma} = \frac{1}{\alpha}  +
	\frac{1}{\beta}$
 % and $\gamma\leq\frac{1}{\rhog}$, 
 with $\alpha,\beta >0$.\cpsolved{I think we will have trouble if
		$\alpha$ or $\beta$ is $0$, in which case the argmin becomes
		the whole space, right?}  Then $P_{\Lambda}(y,z) \neq \emptyset$ and
	\begin{equation}
	 P_{\Lambda}(y,z) =   \bigcup_{\substack{x\in \prox_{\alpha f}(z)
 \\ \xi\in \partial p(y) }}   \argmin_{w\in \Re^n}\, \Phi_{\Lambda,\xi} (w;y,z,x) 
	 \label{eq:prox_formula}	
	\end{equation}	
	for all $(y,z)\in \dom (g)\times \Re^n$,
		where $P_{\Lambda}$ is given by \eqref{eq:Plambda}.
\end{lemma}
\begin{proof}
    Given $x \in \prox_{\alpha f} (z)$ and $\xi \in \partial p(y)$, we have that 
    \begin{align*}
        &~\argmin_{w\in\Re^n} ~\Phi_{\Lambda,\xi} (w;y,z,x)\\
		 =&~\argmin _{w\in \Re^n}~ f(x)+h(x) + \lla \nabla f(x) + \nabla h(x), w- x \rla + \frac{1}{2\alpha}\norm{w-x}^2 \\
        & \qquad \qquad \quad + p(y) + \lla \xi , w-y\rla + \frac{1}{2\beta}\norm{w-y}^2 + g(w)   \\
         = &~\argmin _{w\in \Re^n}~ \lla \nabla f(x) + \nabla h(x) , w\rla + \frac{1}{2\alpha}\norm{w}^2 - \frac{1}{\alpha} \lla w, x\rla \\
        & \qquad \qquad \quad + \lla \xi , w\rla + \frac{1}{2\beta}\norm{w}^2 - \frac{1}{\beta}\lla w,y\rla + g(w) \\
        = &~ \argmin _{w\in \Re^n}~ \frac{1}{2\gamma}\norm{w}^2 - \frac{1}{\alpha } \lla x-\alpha (\nabla f(x) + \nabla h(x)), w \rla - \frac{1}{\beta} \lla y-\beta\xi,w\rla + g(w) \\
    =&~  \argmin _{w\in \Re^n}~\frac{1}{2\gamma} \norm{ w -\left( \frac{\gamma}{\alpha}(x-\alpha \nabla f(x) - \alpha \nabla h(x)) + \frac{\gamma}{\beta}(y-\beta \xi)\right)}^2 + g(w) \\
    =&~  \prox_{\gamma g } \left( \frac{\gamma}{\alpha}(x-\alpha \nabla f(x) - \alpha \nabla h(x)) + \frac{\gamma}{\beta}(y-\beta \xi)\right)
    % & = \prox_{\gamma g } \left( \frac{\gamma}{\alpha}(2x-z - \alpha \nabla h(x)) + \frac{\gamma}{\beta}(y-\beta \xi)\right) 
    \end{align*}
    where the third equality follows from the definition of $\gamma$. The claim now follows from \eqref{eq:Plambda2}. 
    % , and the last equality holds since $z=x+\alpha \nabla f(x)$ by \eqref{eq:prox_optimality}. This completes the proof. 
\end{proof}
}

The following lemma provides an upper bound for the merit function 
$V_{\Lambda,\xi}$. 

\begin{lemma}
	\label{lemma:V_upperbound}
	Under the assumptions of \cref{lemma:prox_formula}, we have 
	\begin{equation*}
	V_{\Lambda,\xi}(y,z,x) \leq 	Q_{\alpha (f+h)}(y;x)  +   p(y)  + g(y) 
	\end{equation*}
for all $(y,z)\in \dom (g)\times \Re^n$, $x \in \prox_{\alpha f} (z)$,  $\xi\in \partial p(y)$, $\ds y^\plus \in \argmin_{w\in\Re^n} \Phi_{\Lambda,\xi} (w;y,z, x)$. 
\end{lemma}
\begin{proof}
We immediately get the result by noting from \eqref{eq:envelope} that $V_{\Lambda,\xi}(y,z,x) \leq \Phi_{\Lambda,\xi}(y;y,z,x)$.
\end{proof}
We now compute a lower bound for $V_{\Lambda,\xi}$. We make use of the estimate provided in \cref{lemma:properties_Lsmooth}(b) in the following lemma, inspired by \cite{ThemelisPatrinos2018} when they studied the Douglas-Rachford algorithm. 
\begin{lemma}
	\label{lemma:V_lowerbound}
	Suppose that \cref{assume:problem} holds. 
 Let $\rhof \geq 0$ be such that
	$f$ is \alert{$\rhof$}-weakly convex
and $L\geq \Lf$ be such that $L-\rhof > 0$.\footnote{By
	\cref{assume:problem}(a) and
	\cref{remark:Lsmooth_implies_weaklyconvex}, we see that there
	indeed exists $\rhof\geq 0$ such that
	$f+\frac{\rhof}{2}\norm{\cdot}^2$ is convex.}
	Then
	\begin{align}
        \begin{split}
         & V_{\Lambda,\xi}(y,z,x)  \\
        \geq&  Q_{\alpha (f+h)}(y^\plus;\hat{x})  + p(y^\plus) + g(y^\plus) + \lla \nabla f(\hat{x})-\nabla f(x) - \frac{1}{\alpha}(\hat{x}-x),x-y^\plus \rla \\
		& \quad + \left\langle \nabla h(x)-\nabla h(\hat{x}),y^\plus -\hat{x}\right\rangle + \frac{1}{2(L-\rhof)}\norm{\nabla f(x)-\nabla f(\hat{x})}^2 \\
		& \quad - \left( \frac{\Lh}{2} + \frac{1}{2\alpha}+ \frac{\rhof L}{2(L-\rhof)} \right) \norm{\hat{x}-x}^2  + \frac{1-\beta \rhop}{2\beta} \norm{y^\plus-y}^2
        \end{split}
	\end{align}
	for all $(y,z)\in \dom (g)\times \Re^n$, $x \in \prox_{\alpha f} (z)$, $\hat{x}\in \Re^n$, $\xi\in \partial p(y)$, and  $y^\plus \in \argmin_{w\in\Re^n} \Phi_{\Lambda,\xi} (w;y,z,x)$.
\end{lemma}
\begin{proof}
	 We have from \cref{lemma:properties_Lsmooth}(b) that for all $\hat{x}\in \Re^n$,
	\begin{align}
		Q_{\alpha f}(y^\plus;x)& \geq \left(f(\hat{x}) + \lla \nabla f(\hat{x}),x-\hat{x}\rla + \frac{1}{2(L-\rhof)}\norm{\nabla f(x)-\nabla f(\hat{x})}^2 \right. \notag \\
  & \qquad  \left.-\frac{\rhof L}{2(L-\rhof)}\norm{x-\hat{x}}^2\right) + \lla \nabla f(x),y^\plus-x \rla + \frac{1}{2\alpha} \norm{y^\plus-x}^2 \notag  \\
		& = f(\hat{x}) + \lla \nabla f(\hat{x})-\nabla f(x),x-y^\plus \rla + \lla \nabla f(\hat{x}),y^\plus -\hat{x}\rla   \notag \\
		& \qquad + \frac{1}{2(L-\rhof)}\norm{\nabla f(x)-\nabla f(\hat{x})}^2 -\frac{\rhof L}{2(L-\rhof)}\norm{x-\hat{x}}^2 + \frac{1}{2\alpha} \norm{y^\plus-x}^2 ,\notag 
	\end{align}
	for any $L\geq \Lf$ such that $L>\rhof$. Using the identity $\norm{a-b}^2 -\norm{a-c}^2 = - \norm{b-c}^2 + 2\lla b-c,b-a\rla$ and after some routine calculations, we further obtain
	\begin{align}
	Q_{\alpha f}(y^\plus ; x) &  \geq Q_{\alpha f}(y^+;\hat{x}) + \left\langle \nabla f(\hat{x})-\nabla f(x) - \frac{1}{\alpha}(\hat{x}-x),x-y^\plus \right\rangle 
	    \notag \\
	    & \qquad - \left(\frac{1}{2\alpha}+ \frac{\rhof L}{2(L-\rhof)}\right) \norm{x-\hat{x}}^2  + \frac{1}{2(L-\rhof)}\norm{\nabla f(x)-\nabla f(\hat{x})}^2 . \label{eq:Qf_lowerbound}
	\end{align}
	On the other hand, we have from \cref{lemma:properties_Lsmooth}(a) \memosolved{AT}{\cref{lemma:properties_Lsmooth}(c)?? (a)??}\jhsolved{Fixed this.} that 
	\begin{align}
	& h(x) + \lla \nabla h(x),y^\plus-x\rla \notag  \\
        \geq& \left( h(\hat{x}) - \lla\nabla h(x),\hat{x}-x \rla - \frac{\Lh}{2}\norm{\hat{x}-x}^2 \right) + \lla \nabla h(x),y^\plus - x\rla  \notag \\
	=& h(\hat{x})+ \lla \nabla h(\hat{x}),y^\plus - \hat{x}\rla   + \lla \nabla h(x)-\nabla h(\hat{x}),y^\plus - \hat{x}\rla - \frac{\Lh}{2}\norm{\hat{x}-x}^2  \label{eq:Qh_lowerbound}
	\end{align}
	for all $\hat{x}\in \Re^n$. 
	By \cref{assume:problem}(d) and the fact that 
	\begin{equation}
	\partial \left( \frac{\rhop}{2}\|\cdot\|^2 - p\right) (y) = \rhop y - \partial p (y),
	\label{eq:partial_p}	
	\end{equation}
	which holds by \cite[Exercise 8.8(c)]{RW98}, we have 
	\[ \frac{\rhop}{2}\norm{y^\plus}^2- p(y^\plus) \geq \frac{\rhop}{2}\norm{y}^2- p(y) + \lla \rhop y - \xi , y^\plus - y \rla ,\]
	for any $\xi\in \partial p(y)$.
	This implies that 
        \begin{equation}
            p(y) \geq p(y^\plus) - \lla \xi, y^\plus-y \rla - \frac{\rhop}{2}\norm{y^\plus -y}^2 
            \label{eq:p_descentlemma}
        \end{equation}
    and hence 
		\begin{equation}
		  Q_{\beta p}(y^\plus;y,\xi) \geq  	  \left( p(y^\plus) + \frac{1-\beta \rhop}{2\beta} \norm{y^\plus-y}^2\right) .
		\label{eq:Qpi_lowerbound}
	\end{equation}
	Using the fact that $V_{\Lambda,\xi}(y,z,x) = \Phi_{\Lambda,\xi} (y^+;y,z,x)$  together with the bounds \eqref{eq:Qf_lowerbound}, \eqref{eq:Qh_lowerbound} and \eqref{eq:Qpi_lowerbound}, we obtain the desired conclusion. 
\end{proof}
% Regarding the requirement of $\rhof$ in \cref{lemma:V_lowerbound},
% by \cref{assume:problem}(a) and
% \cref{remark:Lsmooth_implies_weaklyconvex}, we see that there indeed
% exists $\rhof\geq 0$ such that $f+\frac{\rhof}{2}\norm{\cdot}^2$ is
% convex.

We now show that $\{  V_{\Lambda,\xi^k}(y^k,z^k,x^k) \}$ is a nonincreasing sequence for appropriately chosen stepsizes. To simplify the notations, we denote 
\[V_k \coloneqq V_{\Lambda,\xi^k}(y^k,z^k,x^k) = \Phi_{\Lambda,\xi^k}(y^{k+1};y^k,z^k,x^k).\]
In what follows, we discuss the cases $\tau\in (0,1]$, $\tau\in (1,2)$ and $\tau\in [2,\infty)$ separately. For each case, we will show that when $\Lf+\Lh>0$, there exist a function $c(\alpha)$ and a finite interval $I\subseteq (0,\infty)$ such that (i) $c(\alpha)\leq 0$ on $I$, (ii) $c(\alpha)< 0$ on the interior of $I$, and (iii) the inequality 
    \begin{align}
		V_{k-1} - V_k &  \geq  -\frac{c(\alpha)}{2\tau\alpha} \norm{x^k-x^{k-1}}^2  +  \frac{1-\beta \rhop}{2\beta} \norm{y^k-y^{k-1}}^2  
  %+  \frac{1-\gamma \rhog}{2\gamma}\norm{y^k-y^{k+1}}^2
		\label{eq:suff_decrease}
	\end{align}
 holds for all $k$. For the case $\Lf+\Lh=0$, the above inequality
 also holds but with $\alpha=\infty$ and $c(\alpha)$ a
 negative constant, so that the first term on the right-hand side
 vanishes; see \cref{remark:Lf+Lh=0}.
 
\begin{theorem}[Stepsize for $\tau \in \left(0, 1\right\rbrack$]
\label{thm:Vk_nonincreasing_0to1}
	Suppose
\cref{assume:problem} holds and $\Lf+\Lh>0$. 
If $\{ (x^k,y^k,z^k)\}$ is generated by \cref{alg:main} with  $\tau\in
(0,1]$ and $\alpha \in (0,\bar{\alpha}]$,
	\begin{equation}
	\bar{\alpha}	\coloneqq \begin{cases}
		\frac{1}{\Lf+\Lh} & \text{if}~ (2-\tau)\Lf-2\rhof\geq \tau\Lh,\\
		\frac{\tau}{2\eta^*} & \text{otherwise},
	\end{cases}	
        \label{eq:alpha_bar_tau<1}
	\end{equation}
 and $\eta^*$ is the positive root of 
	\begin{equation}
		q(\eta)\coloneqq	2(2-\tau)\eta^2 - \tau((2-\tau)\Lh + \rhof \tau)\eta - \tau(\rhof^2 + \Lf\Lh),
		\label{eq:q_t<1}
	\end{equation}
	then  \eqref{eq:suff_decrease} holds with
	\begin{align}
	\begin{split}
	    & c(\alpha) \\
         \coloneqq &\begin{cases}
			2\Lf (\Lf+\Lh)\alpha^2 +\left( (2-\tau) \Lh - \tau \Lf \right) \alpha - (2-\tau) & \text{if}~(2-\tau)\Lf-2\rhof\geq \tau\Lh,\\
		2\Lf \Lh \alpha^2 + \left( (2-\tau)\Lh + \frac{\rhof \tau (\eta^*+\rhof) }{\eta^*}\right)\alpha- (2-\tau)& \text{otherwise}.
		\end{cases}
	\end{split}
		\label{eq:c(alpha)}
	\end{align}
	In particular, $\{V_k\}$ is nonincreasing if $\alpha \leq \bar{\alpha}$ and $\beta\leq \rhop^{-1}$, 
 % and $\gamma\leq \rhog^{-1}$, 
 and strictly decreasing if at least one holds with strict inequality. 
\end{theorem}

\begin{proof}
	From \cref{lemma:prox_formula}, we know that for each $k$, there exists $\xi^k\in \partial p (y^k)$ such that $y^{k+1} \in \argmin_{w\in \Re^n} \Phi_{\Lambda,\xi^k}(w;y^k,z^k,x^k)$. 
	By \cref{lemma:V_upperbound},
	\begin{equation}
		V_k \leq 	Q_{\alpha (f+h)}(y^k;x^k) +   p(y^k) 
  % - \left(  \frac{1-\gamma \rhog}{2\gamma}\right)\norm{y^k-y^{k+1}}^2 
  + g(y^k) .
		\label{eq:V_upper_k}
	\end{equation}
	On the other hand, setting $(y,z,x)=(y^{k-1},z^{k-1},x^{k-1})$ and $\hat{x}=x^k$ in \cref{lemma:V_lowerbound}, we have
		\begin{align}
	   \begin{split}
	     V_{k-1} & \geq  Q_{\alpha (f+h)}(y^k;x^k)   + p(y^k) + g(y^k) \\
        & \quad + \left\langle \nabla f(x^k)-\nabla f(x^{k-1}) - \frac{1}{\alpha}(x^k-x^{k-1}),x^{k-1}-y^k \right\rangle \\
		& \quad + \left\langle \nabla h(x^{k-1})-\nabla h(x^k),y^k -x^k\right\rangle + \frac{1}{2(L-\rhof)}\norm{\nabla f(x^{k-1})-\nabla f(x^k)}^2 \\
		& \quad - \left( \frac{\Lh}{2} + \frac{1}{2\alpha}+ \frac{\rhof L}{2(L-\rhof)} \right) \norm{x^k-x^{k-1}}^2  + \frac{1-\beta \rhop}{2\beta} \norm{y^k-y^{k-1}}^2.
	   \end{split}	
    \label{eq:V_lower_k}
	\end{align}
    Subtracting \eqref{eq:V_upper_k} from \eqref{eq:V_lower_k}, we get
	\begin{align}
			V_{k-1} - V_k & \geq \left\langle \nabla f(x^k)-\nabla f(x^{k-1}) - \frac{1}{\alpha}(x^k-x^{k-1}),x^{k-1}-y^k \right\rangle \notag \\
			& \quad + \left\langle \nabla h(x^{k-1})-\nabla h(x^k),y^k -x^k\right\rangle  + \frac{1}{2(L-\rhof)}\norm{\nabla f(x^{k-1})-\nabla f(x^k)}^2 \notag \\
			& \quad 
			- \left(\frac{\Lh}{2} + \frac{1}{2\alpha}+ \frac{\rhof L}{2(L-\rhof)} \right) \norm{x^k-x^{k-1}}^2  + \frac{1-\beta \rhop}{2\beta} \norm{y^k-y^{k-1}}^2 \label{eq:diff_V}
   % \\
			% & \quad + \left(  \frac{1-\gamma \rhog}{2\gamma}\right)\norm{y^k-y^{k+1}}^2. \label{eq:diff_V}
	\end{align}
Meanwhile, we have from \eqref{eq:z-step} that $y^k = \frac{1}{\tau}(z^k-z^{k-1}) + x^{k-1}$ and from \eqref{eq:x-step} and \eqref{eq:prox_optimality} that 
\begin{equation}
    z^k - z^{k-1} = (x^k+\alpha \nabla f(x^k)) - (x^{k-1}+\alpha \nabla f(x^{k-1})).
    \label{eq:z^k_change}
\end{equation} Thus,
	\begin{align}
	y^k-x^{k-1} & = \frac{1}{\tau}(x^k-x^{k-1}) + \frac{\alpha}{\tau}  (\nabla f(x^k) - \nabla f(x^{k-1})), \label{eq:diff1} 
 \end{align}
 which can also be 
 rewritten as
 	\begin{align}
	y^k-x^k & =  \left( \frac{1}{\tau} - 1\right) (x^k-x^{k-1}) + \frac{\alpha}{\tau}  (\nabla f(x^k) - \nabla f(x^{k-1})). 		\label{eq:diff2}
	\end{align}
With these, algebraic calculations lead to 
	\begin{align}
 \nonumber
	&~\left\langle \nabla f(x^k)-\nabla f(x^{k-1}) - \frac{1}{\alpha}(x^k-x^{k-1}),x^{k-1}-y^k \right\rangle \\ \overset{\eqref{eq:diff1}}{=} &~-\frac{\alpha}{\tau}\norm{\nabla f(x^k)-\nabla f(x^{k-1})}^2	 + \frac{1}{\tau\alpha} \norm{x^k-x^{k-1}}^2 \label{eq:inner_prod_f}
	\end{align}
and 
	\begin{align}
		\left\langle \nabla h(x^{k-1})-\nabla h(x^k) ,y^k -x^k\right\rangle  & \overset{\eqref{eq:diff2}}{=}\left( 1-\frac{1}{\tau}\right) \left\langle \nabla h(x^k)-\nabla h(x^{k-1}) ,x^k-x^{k-1} \right\rangle \notag \\
     & \quad  -\frac{\alpha}{\tau}\left\langle \nabla h(x^k)-\nabla h(x^{k-1}) , \nabla f(x^k) - \nabla f(x^{k-1})\right\rangle. \label{eq:inner_prod_h} \\
     & \geq \frac{\tau -1}{\tau}\norm{\nabla h(x^k)-\nabla h(x^{k-1})}\cdot \norm{x^k-x^{k-1}}\notag \\
     & \quad -\frac{\alpha}{\tau} \norm{\nabla h(x^k)-\nabla h(x^{k-1})}\cdot \norm{\nabla f(x^k)-\nabla f(x^{k-1})}\notag \\
     & \geq 	\frac{\tau-1}{\tau}\Lh \norm{x^k-x^{k-1}}^2 - \frac{\alpha}{\tau}\Lf\Lh \norm{x^k-x^{k-1}}^2,\label{eq:inner_prod_h2}
	\end{align}
where the first inequality is by the Cauchy-Schwarz inequality and
noting that $\tau\in (0,1]$, while the last inequality holds by the Lipschitz continuity of the gradients of $f$ and $h$.
Combining \eqref{eq:diff_V}, \eqref{eq:inner_prod_f} and \eqref{eq:inner_prod_h2}, we obtain 
\begin{align}
	V_{k-1} - V_k &  \geq \left( -\frac{\alpha}{\tau}+ \frac{1}{2(L-\rhof)}\right) \norm{\nabla f(x^{k-1})-\nabla f(x^k)}^2 \notag \\
	& \quad 
	+ \left[ \left(\frac{1}{\tau}-\frac{1}{2}\right) \frac{1}{\alpha}-\frac{\Lh}{2} - \frac{\rhof L}{2(L-\rhof)} +\frac{\tau-1}{\tau}\Lh - \frac{\alpha}{\tau}\Lh \Lf\right]\norm{x^k-x^{k-1}}^2  \notag \\ 
	& \quad +  \frac{1-\beta \rhop}{2\beta} \norm{y^k-y^{k-1}}^2.
	% +\left(  \frac{1-\gamma
	% \rhog}{2\gamma}\right)\norm{y^k-y^{k+1}}^2. 
 \label{eq:diff_V2}
\end{align}
Now, we discuss two disjoint cases.
\begin{description}
	\item[Case 1.] Suppose that $(2-\tau)\Lf-2\rhof\geq \tau\Lh$. Then
		$\Lf-\rhof\geq \frac{\tau}{2}( \Lf +  \Lh)>0$ and we may take
		$L=\Lf$ in \eqref{eq:diff_V2}. Let $\hat{\alpha}>0$ be such that $\hat{\alpha}\geq\frac{\tau}{2(\Lf-\rhof)}$, and suppose that $0<\alpha \leq \hat{\alpha}$. Then 
	  \ifdefined\submit 
       $	  	-\frac{\alpha}{\tau}+ \frac{1}{2(\Lf-\rhof)}\geq -\frac{\hat{\alpha}}{\tau}+ \frac{1}{2(\Lf-\rhof)}  ,$
       \else 
        \begin{equation*}
	  	-\frac{\alpha}{\tau}+ \frac{1}{2(\Lf-\rhof)}\geq -\frac{\hat{\alpha}}{\tau}+ \frac{1}{2(\Lf-\rhof)}  ,
	  	\label{eq:lowerbound_coeff_nablaf}
	  \end{equation*}
        \fi 
	where the quantity on the right-hand side is at most zero. Together with the Lipschitz continuity of $\nabla f$ and \eqref{eq:diff_V2} with $L=\Lf$,  for any $\alpha \in (0, \hat{\alpha}]$ we have 
	\begin{align}
		V_{k-1} - V_k
  %       &  \geq \left( -\frac{\hat{\alpha}}{\tau}+ \frac{1}{2(\Lf-\rhof)}\right)\Lf^2 \norm{x^k-x^{k-1}}^2 \notag \\
		% & \quad 
		% + \left[ \frac{2-\tau}{2\tau\hat{\alpha}} -\frac{\Lh}{2} - \frac{\rhof \Lf}{2(\Lf-\rhof)} +\frac{\tau-1}{\tau}\Lh - \frac{\hat{\alpha}}{\tau}\Lh \Lf\right] \norm{x^k-x^{k-1}}^2  \notag \\ 
		% & \quad +  \frac{1-\beta \rhop}{2\beta} \norm{y^k-y^{k-1}}^2  +\left(  \frac{1-\gamma \rhog}{2\gamma}\right)\norm{y^k-y^{k+1}}^2\notag \\
        & \geq \left(\frac{2-\tau}{2\tau\hat{\alpha} } - \hat{\alpha} \frac{\Lf(\Lf+\Lh)}{\tau}  + \frac{\Lf}{2}-\frac{\Lh}{2} + \frac{\tau-1}{\tau}\Lh \right)\norm{x^k-x^{k-1}}^2 \notag \\
		& \quad +  \frac{1-\beta \rhop}{2\beta} \norm{y^k-y^{k-1}}^2
        % +\left(  \frac{1-\gamma \rhog}{2\gamma}\right)\norm{y^k-y^{k+1}}^2
        \notag \\
		& = -\frac{c(\hat{\alpha})}{2\tau\hat{\alpha}} \norm{x^k-x^{k-1}}^2  +  \frac{1-\beta \rhop}{2\beta} \norm{y^k-y^{k-1}}^2  \label{eq:diff_V3} 
        % & \quad +\left(  \frac{1-\gamma \rhog}{2\gamma}\right)\norm{y^k-y^{k+1}}^2 ,\label{eq:diff_V3}
	\end{align}
	where 
	\begin{align}
\nonumber
		    c(\hat{\alpha}) & \coloneqq 2\Lf (\Lf+\Lh)\hat{\alpha}^2 +\left( (2-\tau) \Lh - \tau \Lf \right) \hat{\alpha} - (2-\tau)\notag \\
		& = 
		\left( \hat{\alpha}-\frac{1}{\Lf+\Lh}\right) \left(2\Lf (\Lf+\Lh)\hat{\alpha}+(2-\tau)(\Lf+\Lh) \right) 
        \label{eq:c_alpha_factors},
	\end{align}
	which is nonpositive when $\hat{\alpha}\leq \frac{1}{\Lf+\Lh}$. Hence, \eqref{eq:diff_V3} holds with nonpositive $c(\hat{\alpha})$ when $\hat{\alpha}\in \left[ \frac{\tau}{2(\Lf-\rhof)}, \frac{1}{\Lf+\Lh}\right]$, which is a nonempty interval due to our hypothesis that  $(2-\tau)\Lf-2\rhof\geq \tau\Lh$. %In particular, we may choose the maximum allowable stepsize $\hat{\alpha}=\frac{1}{\Lf+\Lh}$.  %Note that 
 It is clear that $c(\alpha)<0$ when $\alpha < \bar{\alpha}$.

	\item[Case 2.] Suppose now that  
	\begin{equation}
	(2-\tau)\Lf-2\rhof< \tau\Lh.	
	\label{eq:assume_case2}
	\end{equation}
	Given $L\geq \Lf$ with $L-\rhof>0$, we define $\hat{\alpha}(L)
	\coloneqq \frac{\tau}{2(L-\rhof)}$ and select $\alpha \in
(0,\hat{\alpha}(L)]$.
Then
    \ifdefined\submit
    $-\frac{\alpha}{\tau}+ \frac{1}{2(L-\rhof)}\geq -\frac{\hat{\alpha}(L)}{\tau}+ \frac{1}{2(L-\rhof)}  =0.$
    \else 
    \begin{equation*}
		-\frac{\alpha}{\tau}+ \frac{1}{2(L-\rhof)}\geq -\frac{\hat{\alpha}(L)}{\tau}+ \frac{1}{2(L-\rhof)}  =0.
	\end{equation*}
    \fi 
	Hence, we have from \eqref{eq:diff_V2} that
	\begin{align}
		V_{k-1} - V_k 
  % &  \geq 
		%  \left[ \frac{2-\tau}{2\tau\hat{\alpha}(L)} -\frac{\Lh}{2} - \frac{\rhof L}{2(L-\rhof)} +\frac{\tau-1}{\tau}\Lh - \frac{\hat{\alpha}(L)}{\tau}\Lh \Lf\right]\norm{x^k-x^{k-1}}^2  \notag \\ 
		% & \qquad +  \frac{1-\beta \rhop}{2\beta} \norm{y^k-y^{k-1}}^2  +\left(  \frac{1-\gamma \rhog}{2\gamma}\right)\norm{y^k-y^{k+1}}^2 \notag \\
		& \geq  -\frac{c(\hat{\alpha}(L))}{2\tau\hat{\alpha}(L)}\norm{x^k-x^{k-1}}^2  + \frac{1-\beta \rhop}{2\beta} \norm{y^k-y^{k-1}}^2 \label{eq:diff_V_case2} 
  % \\
  %       & \quad +\left(  \frac{1-\gamma \rhog}{2\gamma}\right)\norm{y^k-y^{k+1}}^2 ,\notag \label{eq:diff_V_case2}
	\end{align}
	where 
	\begin{equation}
		c(\hat{\alpha}(L)) \coloneqq 2\Lf \Lh\hat{\alpha}(L)^2 + \left( (2-\tau)\Lh + \frac{\rhof \tau L }{L-\rhof}\right)\hat{\alpha}(L) - (2-\tau).
		\label{eq:ctau_case2}
	\end{equation}
	To determine the largest allowable stepsize $\hat{\alpha}(L)$ so that $c(\hat{\alpha})\leq 0$, we calculate
	\begin{equation}
		L^*\coloneqq \min\{ L : c(\hat{\alpha}(L))\leq 0 ~, L\geq \Lf, L>\rhof\},
		\label{eq:Lstar}
	\end{equation}
	so that $\hat{\alpha}(L^*)$ is the desired stepsize. By some routine calculations, it can be shown that $c(\hat{\alpha}(L)) = -\frac{1}{2(L-\rhof)^2}q(L-\rhof) = -\frac{1}{2\eta^2}q(\eta)$,  where $q$ is given by the polynomial \eqref{eq:q_t<1} and $\eta \coloneqq L-\rhof$. Hence, if $\eta^*$ is the (strictly) positive root of $q$, then $L^* \coloneqq \max\{ \eta^*+\rhof, \Lf \}$. We now claim that $L^* = \eta^*+\rhof$. If $\Lf=\rhof$, this immediately holds since $\eta^*>0$. Suppose now that $\Lf > \rhof$. By the definition of $\hat{\alpha}(L)$, note that we may write $c$ as 
	\begin{equation*}
	c(\hat{\alpha}(L)) = c(\hat{\alpha}(L)) + 2L \Lf \hat{\alpha}(L)^2 - \frac{\tau L \Lf}{L-\rhof}\hat{\alpha}(L)
	\end{equation*}
	for any $L>\rhof$. 
	Simplifying this expression, we obtain
	\begin{equation}
		c(\hat{\alpha}(L)) =2\Lf (\Lh+L)\hat{\alpha}(L)^2 +\left( (2-\tau) \Lh - \frac{\tau L(\Lf-\rhof)}{L-\rhof} \right) \hat{\alpha}(L) - (2-\tau).
		\label{eq:calpha_alternative}
	\end{equation}
	Since $\Lf>\rhof$, $c(\hat{\alpha}(\Lf))$ is well-defined and can be calculated as 
	\begin{align*}
		c(\hat{\alpha}(\Lf)) & = 2\Lf (\Lf+\Lh)\hat{\alpha}(\Lf)^2 +\left( (2-\tau) \Lh - \tau \Lf\right) \hat{\alpha}(\Lf) - (2-\tau) \\
		& = \left( \hat{\alpha}(\Lf)-\frac{1}{\Lf+\Lh}\right) \left(2\Lf (\Lf+\Lh)\hat{\alpha}(\Lf)+(2-\tau)(\Lf+\Lh) \right) .
	\end{align*}
	 Since \eqref{eq:assume_case2} implies that $\hat{\alpha}(\Lf) >
	 \frac{1}{\Lf+\Lh}$, it follows that $c(\hat{\alpha}(\Lf))>0$.
	 Hence, $L^*\neq \Lf$ by the definition of $L^*$ in
	 \eqref{eq:Lstar}. The claim that $L^*=\eta^*+\rhof$ now follows.
	 We also note that since $c(\hat{\alpha}(L^*))\leq 0$, it follows
	 from \eqref{eq:ctau_case2} that $c(\alpha)<0$ for any $\alpha <
	 \hat{\alpha}(L^*)$, where $c(\alpha)$ is as defined in
	 \eqref{eq:c(alpha)}. 	\ifdefined\submit
		\else
		\qedhere
		\fi
\end{description}
\end{proof}

\begin{remark}
\label{remark:case2_magnitude}  To gain insight on the magnitude of the stepsize upper bound $\bar{\alpha}=\frac{\tau}{2\eta^*}$ in the second case of the above proof, consider $L\coloneqq \frac{\tau \Lh + 2\rhof}{2-\tau}$. By \eqref{eq:assume_case2}, $L>\Lf$ and $L>\rhof$. In addition, for this choice of $L$, $c(\hat{\alpha}(L))\leq 0$. By \eqref{eq:Lstar}, it holds that $L\geq L^*$, ensuring  $\frac{\tau}{2\eta^*} \geq \frac{\tau}{2(L-\rhof)} = \frac{2-\tau}{2(\Lh +\rhof)}$. To obtain an upper bound, we note that from \eqref{eq:calpha_alternative}, it can be verified that an alternative way to express $c(\hat{\alpha}(L))$ is
\begin{align*}
c	(\hat{\alpha}(L)) & =\left[ 2\Lf (\Lf+\Lh) \hat{\alpha}(L)^2+ ((2-\tau)\Lh - \tau \Lf)\hat{\alpha}(L) - (2-\tau) \right] \\
& \qquad + 2\Lf (L-\Lf) \hat{\alpha}(L)^2 +\left( \tau\Lf - \frac{\tau L(\Lf-\rhof)}{\Lf-\rhof} \right) \hat{\alpha}(L) \\
& = \left( \hat{\alpha}(L)-\frac{1}{\Lf+\Lh}\right) \left(2\Lf (\Lf+\Lh)\hat{\alpha}(L)+(2-\tau)(\Lf+\Lh) \right) \\
&\qquad + 2\Lf (L-\Lf)\hat{\alpha}(L)^2 + \frac{\tau\rhof(L-\Lf)}{L-\rhof}\hat{\alpha}(L).
\end{align*}
Note that the last two terms are nonnegative when $L=L^*$ since
$L^*\geq \Lf$ and $L^*> \rhof$. Since $c(\hat{\alpha}(L^*))= 0$, it
follows that $\hat{\alpha}(L^*)-\frac{1}{\Lf+\Lh}\leq 0$. That is, $\frac{\tau}{2\eta^*}\leq \frac{1}{\Lf+\Lh}$. In summary, when $\tau \in (0,1]$, we have
 $\frac{2-\tau}{2(\Lh+\rhof)}\leq \frac{\tau}{2\eta^*}\leq
\frac{1}{\Lf+\Lh}$  and therefore
\begin{equation}
	\bar{\alpha} \leq \frac{1}{L_f + L_h}.
	\label{eq:alphabound}
\end{equation}
\end{remark}

\begin{remark}[Stepsize comparison with \citep{BianZhang21}]
\label{remark:DYS_stepsize} For the case $\tau=1$, the above theorem implies that strict
monotonicity of $\{V_k\}$ holds when $\alpha<\bar{\alpha}$, where
\begin{equation*}
     \bar{\alpha} = \begin{cases}
    \frac{1}{\Lf+\Lh} & \text{if}~ \Lf - 2\rhof \geq \Lh, \\
    \frac{2}{\Lh + \rhof + \sqrt{(\Lh + \rhof)^2 + 8 (\rhof^2 + \Lf\Lh)}} & \text{otherwise}.   
\end{cases}
\end{equation*}
On the other hand, the bound derived in \cite[Lemma 3.3]{BianZhang21}
for the DYS algorithm (\textit{i.e.}, $p\equiv 0$) indicates that the stepsize $\alpha>0$ should satisfy 
\[ \frac{1}{2} \left( \frac{1}{\alpha}-\rhof\right) -\Lh - \left(  \frac{1}{\alpha}+\frac{\Lh}{2}\right)(2\alpha\rhof + 2\alpha \Lf + \alpha ^2\Lf^2) > 0,\]
\jhsolved{Thanks for checking my calculations. And sorry for the typo here.
The condition should be $>0$ (in the above line). The equivalence of
the two holds by the positivity of $\alpha$ (which I now emphasized).
Now, I guess with these, the original $d(\alpha)$ (with the signs I
put) is correct now? What do you think? And about why this has exactly
only one positive root, it follows by the Descartes' Rule of Signs.
Let me know if this fixes the problem. I will also check it again in
more detail a bit later. Thank you!}
\cpsolved{Verified the rest with no problem, and thanks for the hint
of Descartes' rule of signs.}
or equivalently, under the constraint $\alpha > 0$,
\cpsolved{My calculation has several places with a different sign,
please double check (and also the computations after that). This makes
more sense to me: if the leading coefficient was positive, it would be more
likely that when $\alpha$ approaches infinity, the polynomial is
positive but not negative. If that's indeed the case, all the
deliberate calculations from here on would be invalid? Anyway
therefore I didn't check further calculations in this remark.}
\[ d(\alpha) \coloneqq \Lf^2\Lh \alpha^3  + 2(\Lf^2 +\Lh \Lf + \rhof
	\Lh)\alpha^2 +  (5\rhof + 2\Lh + 4\Lf)\alpha
- 1 <0.\]
Hence, the upper bound for $\alpha$ is the unique positive root
$\hat{\alpha}$ of the polynomial $d$ given above.\cpsolved{Didn't verify if
	there is only one positive root. But if that's true, can we really
	say those 2 are equivalent, given that we multiply the first by
	the variable $2\alpha$? One simple example: consider $p_1(x) = x^{-1} + 1 +
	x + x^2$ and $p_2(x) = x p_1(x) = 1 + x + x^2 + x^3$. For the
former we have $p_1(-2) > 0$ and thus $p_2(-2) < 0$. This could give
different intervals between the two I guess.}\jhsolved{In the particular example you gave, the reason is that the chosen $x$ is negative. Indeed, given $x>0$, $xp(x)>0$ is equivalent to $p(x)>0$.}
Consider $c(\alpha)$ given in \eqref{eq:c(alpha)}. In the first case,
that is, when $\Lf - 2\rhof \geq \Lh$,
as long as $\Lf+\Lh >0$ and $\alpha > 0$,
we have
\[ c(\alpha) - d(\alpha) = -\Lf^2\Lh \alpha^3 - 2\rhof\Lh \alpha^2 - (5\rhof+\Lh+5\Lf) <0.\]
Hence, $-d(\bar{\alpha}) = c(\bar{\alpha}) - d(\bar{\alpha}) <0$, and since $\hat{\alpha}$ is the unique positive root of $d$, it follows that $\hat{\alpha}< \bar{\alpha}$. In the second case, recall that $\bar{\alpha} = \hat{\alpha}(L^*) = \frac{1}{2\eta^*} = \frac{1}{2(L^*-\rhof)}$, so that 
\begin{align*}
    c(\bar{\alpha}) & = c(\bar{\alpha})  + 2\rhof \Lf \bar{\alpha}^2 - \frac{\rhof \Lf}{L^*-\rhof}\bar{\alpha} \\
    & = 2(\Lf\Lh + \rhof \Lf )\bar{\alpha}^2 + \left( \Lh + \frac{\rhof (L^*-\Lf)}{L^*-\rhof} \right) \bar{\alpha} - 1 \\ 
    & \leq 2(\Lf\Lh +  \Lf^2 )\bar{\alpha}^2 + \left( \Lh + \rhof \right) \bar{\alpha} - 1,
\end{align*}
where the last inequality holds since $\rhof\leq \Lf$ and $L^* > \rhof$. Then,
provided that $\Lf+\Lh>0$, we have
\begin{align*}
    c(\bar{\alpha})  - d(\bar{\alpha})  \leq -\Lf^2\Lh \bar{\alpha}^3 -2\rhof \Lh \bar{\alpha}^2 - (4\rhof +\Lh + 4\Lf) < 0,
\end{align*}
and similar to the previous case, we get that $d(\bar{\alpha}) >0$ and
therefore $\hat{\alpha}< \bar{\alpha}$. This shows that our stepsize
upper bound is always larger than that in \citep{BianZhang21}. The
significant gap between the computed stepsizes is evident in
\cref{fig:compare_stepsize1}.
\alert{As we will see in the experiment in \cref{sec:stepsize}, a larger $\alpha$ leads to
faster convergence.
Our improved upper bound for $\alpha$ therefore also provides
numerical acceleration for DYS.}
\end{remark}

\ifdefined\submit
  \def\figscale{0.35}
\else
  \def\figscale{.75}
\fi

 \begin{figure}[tb]
      \centering
    \includegraphics[scale=\figscale]{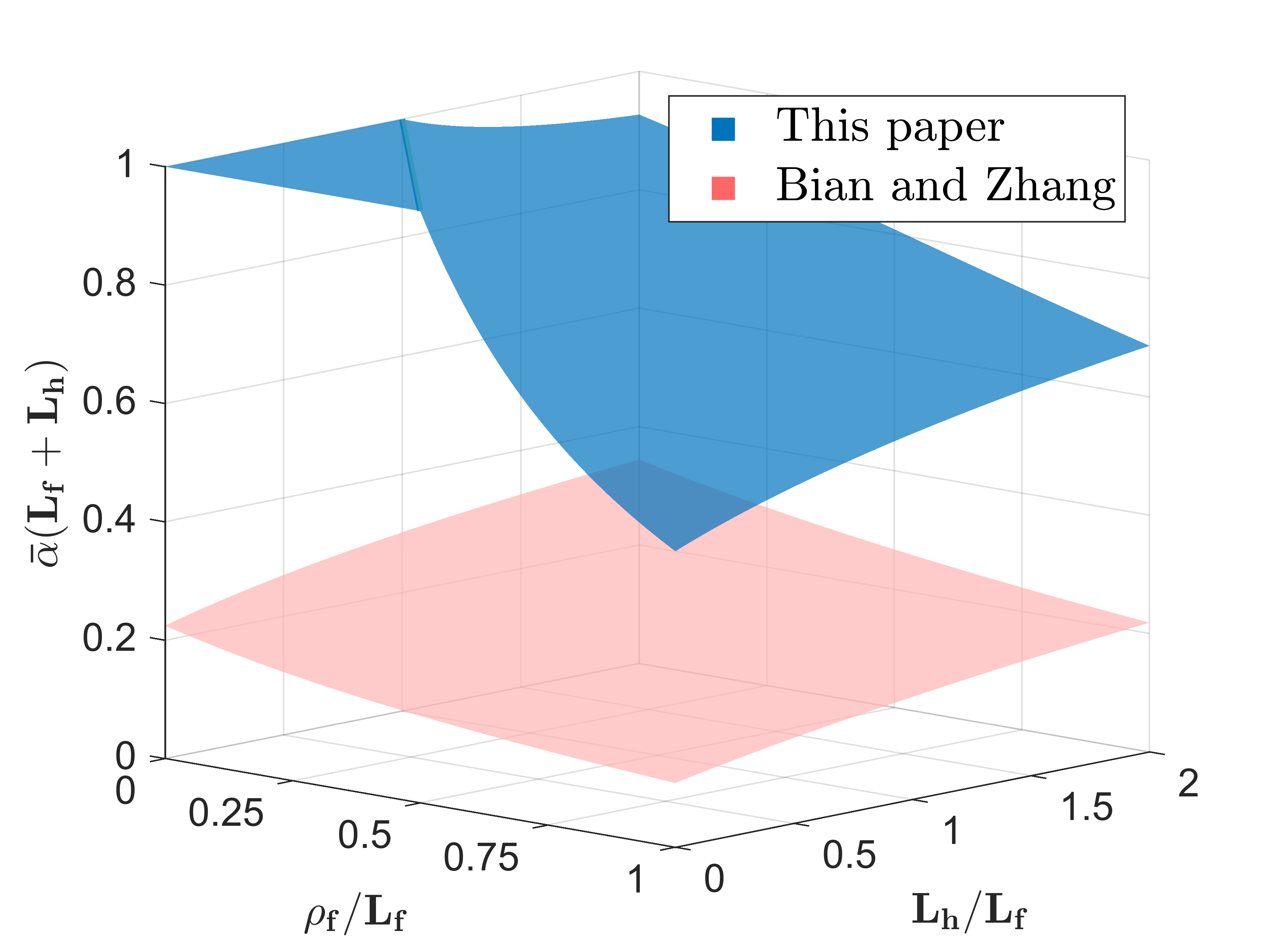}
     \caption{Comparison of stepsize upper bounds (denoted by
 		$\bar{\alpha}$) for the DYS algorithm in this paper
 		(with $\tau = 1$) and in \citep{BianZhang21}.}
     \label{fig:compare_stepsize1}
 \end{figure}

\begin{remark}[The case $\Lf+\Lh=0$]
\label{remark:Lf+Lh=0}
    Suppose that $\Lf+\Lh=0$, in which case $\alpha=\infty$ and $\gamma=\beta$. From \eqref{eq:diff_V}, we immediately obtain 
        \begin{align*}
		V_{k-1} - V_k &  \geq  \frac{1-\beta \rhop}{2\beta} \norm{y^k-y^{k-1}}^2  ,
  % +  \frac{1-\gamma \rhog}{2\gamma}\norm{y^k-y^{k+1}}^2 ,
		\label{eq:suff_decrease_Lf+Lh=0}
	\end{align*}
 for any $\tau>0$ (Note that in this case, the $x$ and $z$ sequences generated by \cref{alg:main} are irrelevant). Thus, we still obtain the desired inequality \eqref{eq:suff_decrease} by setting $c(\alpha)$ to be any negative number.  
\end{remark}

The following provides a result similar to the previous theorem but for parameters $\tau$ in $(1,2)$. Note that the parameter  $\sigmah$ in the theorem below exists; for instance, we can take $\sigmah = -\Lh$ by \cref{lemma:properties_Lsmooth}(a). 
\begin{theorem}[Stepsize for $\tau\in (1,2)$]
\label{thm:Vk_nonincreasing_1to2}
	Suppose that \cref{assume:problem} holds and $\Lf+\Lh>0$. Let
	$\sigmah\in \Re$ be such that $h-\frac{\sigmah}{2}\norm{\cdot}^2$ is convex.\footnote{ We note that there indeed exists $\sigmah\in \Re$ such that $h-\frac{\sigmah}{2}\norm{\cdot}^2$ is convex; for instance, we can take $\sigmah = -\Lh$ by \cref{lemma:properties_Lsmooth}(a).
 }
	 If $\{ (x^k,y^k,z^k)\}$ is generated by \cref{alg:main} with
 $\tau \in (1,2)$ and  $\alpha \in (0,  \bar{\alpha}]$, where 
 \[
	 \bar{\alpha} \coloneqq \begin{cases}
	     \bar{\alpha}_1 & \text{if}~ \tau \leq 2\bar{\alpha}_1(\Lf-\rhof) \\
        \frac{\tau}{2\eta^*} & \text{otherwise}	 \end{cases},
\]
   $\bar{\alpha}_1$ is the positive root of 
	\begin{equation}
		c(\alpha) \coloneqq 2\Lf (\Lf+\Lh)\alpha^2 +\left( \tau \Lh -2(\tau-1)\sigmah - \tau \Lf \right) \alpha - (2-\tau), 
		\label{eq:c_alpha_tau>1}
	\end{equation}
	and $\eta^*$ is the positive root of 
	\begin{equation*}
		q(\eta) \coloneqq 2(2-\tau)\eta^2 -\tau (\tau \Lh -2(\tau-1)\sigmah +\rhof \tau)\eta - \tau^2 (\rhof^2+\Lf\Lh),
		\label{eq:q_t>1}
	\end{equation*} 	
  then \eqref{eq:suff_decrease} holds 
	with $c(\alpha)$ given by \eqref{eq:c_alpha_tau>1} if $\tau \leq 2\bar{\alpha}_1(\Lf-\rhof)$, and 
	\[c(\alpha)= 2\Lf \Lh \alpha^2 + \left(  \tau \Lh -2(\tau-1)\sigmah+ \frac{\rhof \tau (\eta^*+\rhof) }{\eta^*}\right)\alpha- (2-\tau))\]
	otherwise. 
	In particular, $\{V_k\}$ is nonincreasing if $\alpha \leq \bar{\alpha}$ and $\beta\leq \rhop^{-1}$, 
 % and $\gamma\leq \rhog^{-1}$, 
 and strictly decreasing if at least one holds with strict inequality. 
	
\end{theorem}
 
 \begin{proof}
 Given  $\tau\geq 1$, we have from \eqref{eq:inner_prod_h} and the monotonicity of $\nabla h -\sigmah Id$ that  
 		\begin{equation}
	 	\left\langle \nabla h(x^{k-1})-\nabla h(x^k) ,y^k -x^k\right\rangle \geq  \left( \frac{\tau-1}{\tau}\right)\sigmah \norm{x^k-x^{k-1}}^2 - \frac{\alpha}{\tau}\Lf\Lh \norm{x^k-x^{k-1}}^2. 
	 	\label{eq:inner_prod_h>1}
 		\end{equation}
 	The rest of the proof follows arguments similar to those in the proof of \cref{thm:Vk_nonincreasing_0to1}. 
 \end{proof}
 % We note that there indeed exists $\sigmah\in \Re$ such that $h-\frac{\sigmah}{2}\norm{\cdot}^2$ is convex; for instance, we can take $\sigmah = -\Lh$ by \cref{lemma:properties_Lsmooth}(a).

 \begin{remark} 
 \label{remark:tau>1_samebounds}
From \cref{thm:Vk_nonincreasing_1to2}, it can be shown  that the
stepsize bounds for $\tau\in (1,2)$ are also given by
\eqref{eq:alpha_bar_tau<1} if $\sigmah=\Lh$. In particular, when
$h=0$, we obtain that $\bar{\alpha} = \min \left\lbrace \frac{1}{\Lf},
\frac{2-\tau}{2\rhof}\right\rbrace$ for any $\tau\in (0,2)$, which is
the same as the one obtained in \cite[Theorem
4.1]{ThemelisPatrinos2018}, where the case $h=p=0$ is considered.
Another scenario for $\sigmah=\Lh$ is when $h$ is a quadratic
function with its Hessian being a positive multiple of the identity
matrix.
 \end{remark}
 
 \begin{remark}
 	\label{remark:case2_magnitude_>1}
	Similar to \cref{remark:case2_magnitude}, the bound
	\cref{eq:alphabound} also holds for $\tau \in (1,2)$.
	To see this, note that $c(\alpha)$ in
	\eqref{eq:c_alpha_tau>1} can also be written as 
    \ifdefined\submit
    $c(\alpha) =  \hat{c}(\alpha) + 2(\tau-1)(\Lh -\sigmah)\alpha, $
    \else 
 	\[ c(\alpha) =  \hat{c}(\alpha) + 2(\tau-1)(\Lh -\sigmah)\alpha, \]
  \fi
 	where $\hat{c}(\alpha) \coloneqq 2\Lf (\Lf+\Lh)\alpha^2 +\left(
	(2-\tau) \Lh - \tau \Lf \right) \alpha- (2-\tau)$. Since
	$\frac{1}{\Lf+\Lh}$ is the positive root of $\hat{c}(\alpha)$ (see
	\eqref{eq:c_alpha_factors}) and $\Lh -\sigmah \geq 0$ by
	\cref{remark:Lsmooth_implies_weaklyconvex},
	it holds that $c\left( \frac{1}{\Lf+\Lh}\right) =
	\frac{2(\tau-1)(\Lh-\sigmah)}{\Lf+\Lh}\geq 0$. Since $c(0) =
	\tau-2 <0$ and $\bar{\alpha}_1$ is the positive root of
	$c(\alpha)$, it follows that $\bar{\alpha}_1\leq
	\frac{1}{\Lf+\Lh}$. On the other hand, following the same
	arguments in \cref{remark:case2_magnitude}, it can be shown that
	$\tau / (2 \eta^*) \leq \bar{\alpha}_1$, so \cref{eq:alphabound}
	holds as claimed. 
 \end{remark}

 Lastly, we derive stepsize upper bounds for $\tau\geq 2$.
 \begin{theorem}[Stepsize for $\tau\in [2,\infty)$]	\label{thm:Vk_nonincreasing_>2}
 	Suppose that \cref{assume:problem} holds and $\rhoh\in
	[0,\Lh]$ such that $h+\frac{\rhoh}{2}\norm{\cdot}^2$ is convex.\footnote{By \cref{assume:problem}(a) and
\cref{remark:Lsmooth_implies_weaklyconvex}, we see that there indeed
exists $\rhoh\geq 0$ such that $h+\frac{\rhoh}{2}\norm{\cdot}^2$ is
convex.} In addition, suppose that $f$ is
	$\sigmaf$-strongly convex for some $\sigmaf > 0$. \cpsolved{Changed the description so that writing $\sigmaf$ here is more for notational definition. This is to be consistent with our modification in the subsequential convergence result. Does this look ok?}\jhsolved{This looks okay, but I guess we can explicitly state that $\sigmaf>0$ instead of $\sigma\in \Re$ (or probably write this in the first sentence as an implication of \eqref{eq:delta2})? }
 \cpsolved{Then let me revert back to the original setting.}
	Let $\tau \geq 2$ be such that 
 	\begin{equation}
 		\delta(\tau) \coloneqq (\tau \nu - \tau \theta_1 - 2(\tau-1)\theta_2)^2 - 8(\theta_0+\nu) (\tau - 2) > 0,
 		\label{eq:delta}
 	\end{equation}
 	and
        \begin{equation}
        \tau \nu - \tau \theta_1 - 2(\tau-1)\theta_2 > 0,
        \label{eq:delta2}
        \end{equation}
        where $\nu\coloneqq \frac{\sigmaf}{\Lf+\Lh}$,\cpsolved{Usually I
			think $\kappa$ is used for the condition number, which
			would be the reciprocal of this definition. Is there some
		specific reason that you use this notation?}\jhsolved{Oh, it must be because of my earlier notations in my computations that I indeed used the reciprocal then, but then I forgot to change the notation during revision. So we can change this notation to something else.} $\theta_0\coloneqq \frac{\Lh(\Lf^2-\sigmaf^2)}{\Lf(\Lf+\Lh)^2}$, $\theta_1 \coloneqq \frac{\Lh}{\Lf+\Lh}$ and $\theta_2 \coloneqq \frac{\rhoh}{\Lf+\Lh}$, so that
 	\begin{equation}
 		r(\mu) \coloneqq \tau^2(\theta_0+\nu) \mu^2 - \tau^2 \left(\nu - \theta_1 - \frac{2(\tau-1)}{\tau}\theta_2 \right) \mu + 2(\tau-2)
 		\label{eq:r_mu}
 	\end{equation}
 	has two distinct roots $\mu_*$ and $\mu^*$ with $0\leq \mu_*<
	\mu^* \leq 1$. If $\{ (x^k,y^k,z^k)\}$ is generated by
	\cref{alg:main} with stepsize $\alpha = \frac{\tau
	\mu}{2(\Lf+\Lh)}$, where $\mu >0$ satisfies $\mu \in [\mu_*,\mu^*]$, then \eqref{eq:suff_decrease} holds 
 	with $c(\alpha)$ given by 
 	\begin{equation}
 		c(\alpha) = 2\Lh\Lf \alpha^2 + \left(\tau \mu \sigmaf  + \tau \Lh + 2(\tau-1)\rhoh - \tau \sigmaf  -\frac{-\tau \Lh \sigmaf^2 \mu}{\Lf(\Lf+\Lh)} \right) \alpha - (2-\tau) .
 		\label{eq:c_alpha_tau>2}
 	\end{equation}
	In particular, $\{V_k\}$ is nonincreasing if $\alpha \leq \bar{\alpha}$ and $\beta\leq \rhop^{-1}$, 
 % and $\gamma\leq \rhog^{-1}$, 
 and strictly decreasing if at least one holds with strict inequality. 
 \end{theorem}
 \begin{proof}
 	$\sigmaf$-strong convexity implies that 
  \[ f(x) \geq f(\hat{x}) + \lla \nabla f(\hat{x}), x-\hat{x}\rla +  \frac{\sigmaf}{2}\norm{x-\hat{x}}^2,\quad \forall x,\hat{x}\in \Re^n.\]
By the above inequality together with the $\Lf$-smoothness of $f$ and \eqref{eq:lowerbound_intermsofgrad}, we have 
 	\[ f(x) \geq f(\hat{x}) + \lla \nabla f(\hat{x}), x-\hat{x}\rla + \frac{\mu}{2\Lf}\norm{\nabla f(x)-\nabla f(\hat{x})}^2 + \frac{(1-\mu)\sigmaf}{2}\norm{x-\hat{x}}^2, \]
 	for any $x,\hat{x}\in \Re^n$ and $\mu\in [0,1]$. Following the arguments in \cref{lemma:V_lowerbound}, we obtain 
 		\begin{align*}
 		& V_{\Lambda,\xi}(y,z,x) \geq \\
           & \quad Q_{\alpha (f+h)}(y^\plus;\hat{x})  + p(y^\plus) + g(y^\plus) + \left\langle \nabla f(\hat{x})-\nabla f(x) - \frac{1}{\alpha}(\hat{x}-x),x-y^\plus \right\rangle \\
 		&\quad + \left\langle \nabla h(x)-\nabla h(\hat{x}),y^\plus -\hat{x}\right\rangle + \frac{\mu}{2\Lf}\norm{\nabla f(x)-\nabla f(\hat{x})}^2 \\
 		&\quad- \left( \frac{\Lh}{2} + \frac{1}{2\alpha}-\frac{(1-\mu)\sigmaf}{2} \right) \norm{\hat{x}-x}^2  + \frac{1-\beta \rhop}{2\beta} \norm{y^\plus-y}^2.
 	\end{align*}
 	By using this bound and inequality \eqref{eq:inner_prod_h>1} with $\sigmah=-\rhoh$, we obtain by following the same arguments and calculations in the proof of \cref{thm:Vk_nonincreasing_0to1} that 
 	\begin{align}
 		V_{k-1} - V_k &  \geq \left( -\frac{\alpha}{\tau}+ \frac{\mu}{2\Lf}\right) \norm{\nabla f(x^{k-1})-\nabla f(x^k)}^2 \notag \\
 		& \quad 
 		+ \left[ \frac{2-\tau}{2\tau\alpha}-\frac{\Lh}{2} + \frac{(1-\mu)\sigmaf}{2} -\frac{\tau-1}{\tau}\rhoh - \frac{\alpha}{\tau}\Lh \Lf\right]\norm{x^k-x^{k-1}}^2  \notag \\ 
 		& \quad +  \frac{1-\beta \rhop}{2\beta} \norm{y^k-y^{k-1}}^2  %+\left(  \frac{1-\gamma \rhog}{2\gamma}\right)\norm{y^k-y^{k+1}}^2. \notag \label{eq:diff_V>2}
 	\end{align}
 	By setting $\alpha = \frac{\tau \mu}{2(\Lf+\Lh)}$ with $\mu\in
(0,1]$,\cpsolved{How did you decide this value of $\alpha$? If you do not
	replace $\alpha$ by this value but simply use strong convexity to
	merge the first two terms, would it be possible to get a better
(though might be uglier) upper bound for $\alpha$?}\jhsolved{Hhmm. In one of the earlier drafts, it is actually set to $\frac{\tau\mu}{2\Lf}$ to get rid of the first term (similar to the technique used in Case 2 of the proof for $\tau\in (0,1]$. But I realized it's quite hard to relate it to $\frac{1}{\Lf+\Lh}$, due to the difference in the denominators, which is important because in \cref{prop:bounded}, we need to have $\alpha<\frac{1}{\Lf+\Lh}$.  So then I decided to set $\alpha$ as a multiple of $\frac{1}{\Lf+\Lh}$, which will make the comparison easier. I think what you said can also be done, but might be tedious and the expressions might indeed be uglier. The calculations will take a lot of work, so I will postpone this task later.}
\cpsolved{Indeed I also figure the calculations would be lengthy and thus opted for asking instead of deriving.. I think this doesn't matter that much at the current moment because this theorem has limited application in real cases. Just skip it would be totally fine.}
	we have
 	\begin{align*}
 		\left( -\frac{\alpha}{\tau}+ \frac{\mu}{2\Lf}\right) \norm{\nabla f(x^{k-1})-\nabla f(x^k)}^2  %& =  \frac{\mu\Lh}{2\Lf (\Lf+\Lh)} \norm{\nabla f(x^{k-1})-\nabla f(x^k)}^2 \\ 
 		%&
   \geq \frac{\mu \Lh\sigmaf^2}{2\Lf (\Lf+\Lh)}
		\norm{x^{k-1}-x^k}^2,
 	\end{align*}
 	where the inequality follows from the strong convexity of $f$.
	Therefore, \eqref{eq:suff_decrease} holds
%  	\[ V_{k-1} - V_k  \geq
% 	-\frac{c(\alpha)}{2\tau\alpha}\norm{x^k-x^{k-1}}^2 +
% \frac{1-\beta \rhop}{2\beta} \norm{y^k-y^{k-1}}^2  +\left(
% \frac{1-\gamma \rhog}{2\gamma}\right)\norm{y^k-y^{k+1}}^2,\]
 	with $c(\alpha)$ given by \eqref{eq:c_alpha_tau>2}. To obtain
	a nonincreasing sequence $\{V_k\}$, we need to find the range of
	$\mu$ that makes $c(\alpha)\leq 0$.  Plugging in $\alpha = \frac{\tau \mu}{2(\Lf+\Lh)}$ in \eqref{eq:c_alpha_tau>2}, some algebraic calculations lead to
    \ifdefined\submit 
    $c(\alpha) = c\left( \frac{\tau \mu}{2(\Lf+\Lh)}\right) = \frac{1}{2}r(\mu), $
     \else 
 	\[ c(\alpha) = c\left( \frac{\tau \mu}{2(\Lf+\Lh)}\right) = \frac{1}{2}r(\mu), \]
  \fi 
 	with $r(\mu)$ given by \eqref{eq:r_mu}.
	We note that $r(\mu)$ has two distinct roots $\mu_*$ and $ \mu^*$
	with $\mu_*< \mu^*$ if and only if its discriminant is positive.
	This condition is equivalent to \eqref{eq:delta}. Meanwhile, since
	\eqref{eq:delta2} holds and $\rhoh\geq 0$, we see that
 	\[ 0< \frac{\tau \nu - \tau \theta_1 - 2(\tau-1)\theta_2}{2\tau (\theta_0 + \nu)}\leq \frac{\tau\nu}{2\tau \nu} = \frac{1}{2}.\]
 That is to say, the first coordinate of the vertex of the parabola
 defined by $r(\mu)$ lies in $(0,\frac{1}{2}]$. Since $\tau\geq 2$, it
 follows that $0\leq \mu_* < \mu^* \leq 1$. Thus, $c(\alpha) \leq 0$ for $\alpha = \frac{\tau \mu}{2(\Lf+\Lh)}$ with $\mu \in [\mu_*,\mu^*]$, and $c(\alpha)<0$ if $\mu \in (\mu_*,\mu^*)$. This completes the proof. 
 \end{proof}

 \begin{remark}[Stepsize for $\tau=2$]
 \label{remark:tau=2}
Suppose that $\tau=2$. Then \cref{eq:delta,eq:delta2} are equivalent
to having $\sigmaf > \Lh + \rhoh$, and the roots of \eqref{eq:r_mu}
are $\mu_* =0$ and $\mu^* = \frac{\Lf (\Lf+\Lh)(\sigmaf - \Lh -
\rhoh)}{\Lh(\Lf^2-\sigmaf^2) + \sigmaf \Lf (\Lf +\Lh)}$. Hence,
$c(\alpha)$ is strictly negative for all $\alpha$ such that \cpsolved{I
	think the value above is correct, but can't see where the one
below comes from.}\jhsolved{When $\tau=2$, the range for $\alpha$ is given
	by $\frac{\mu_*}{2(\Lf+\Lh)} \leq \alpha \leq
	\frac{\mu^*}{\Lf+\Lh}$, according to the theorem. So we just plug
in the computed roots.}\cpsolved{Ah sorry, I forgot $\alpha$ is not simply
$\mu$\ldots my bad.}
\[ 0 < \alpha < \frac{\Lf (\sigmaf - \Lh - \rhoh)}{\Lh(\Lf^2-\sigmaf^2) + \sigmaf \Lf (\Lf +\Lh)}.\]
 \end{remark}
 \begin{remark}
 	If $h\equiv0$, the condition \eqref{eq:delta2} automatically holds since $\sigmaf >0$. Moreover, the condition \eqref{eq:delta} with the constraint $\tau\geq 2$ is equivalent to having 
 	\begin{equation}
 		2\leq\tau < \frac{4}{1+\sqrt{1-\nu}}\quad \text{or} \quad \tau > \frac{4}{1-\sqrt{1-\nu}}.
 		\label{eq:delta_h=0}
 	\end{equation}
 	On the other hand, the roots of \eqref{eq:r_mu} are given by 
 	\[ \mu_* =\frac{1}{2} - \frac{\sqrt{\nu (\nu \tau^2 - 8\tau + 16)}}{2\tau \nu} \quad \text{and} \quad \mu^* =\frac{1}{2} + \frac{\sqrt{\nu (\nu \tau^2 - 8\tau + 16)}}{2\tau \nu} .\]
 	\cref{thm:Vk_nonincreasing_>2} asserts that if we choose $\tau$
	that satisfies \eqref{eq:delta_h=0}, then $\{ V_k\}$ is strictly
	decreasing provided that the stepsize $\alpha$ satisfies
 	\begin{equation}
% 		\frac{\tau}{2\Lf} \left( \frac{1}{2} - \frac{\sqrt{\nu (\nu \tau^2 - 8\tau + 16)}}{2\tau \nu}\right)  < \alpha < \frac{\tau}{2\Lf} \left( \frac{1}{2} + \frac{\sqrt{\nu (\nu \tau^2 - 8\tau + 16)}}{2\tau \nu}\right),
 		\label{eq:alpha_range}
 %	\end{equation}
 %	or equivalently, 
 %	\[
 \frac{\tau \nu -\sqrt{\nu (\nu \tau^2 - 8\tau + 16)}}{4\sigmaf}  < \alpha < \frac{\tau \nu +\sqrt{\nu (\nu \tau^2 - 8\tau + 16)}}{4\sigmaf},
\end{equation}
 	which are the bounds obtained in \cite[Theorem 4.1]{ThemelisPatrinos2018}, where the case $h\equiv p\equiv 0$ is considered. 
 	However, in the said result, the analysis restricts $\tau$ to
	satisfy only the first condition in \eqref{eq:delta_h=0}, due to
	their imposed constraint that $\alpha \Lf $ must be at most 1.
	Meanwhile, the analysis we provide in the proof of
	\cref{thm:Vk_nonincreasing_>2} does not require this condition to
	establish the nonincreasing property, and therefore we have shown
	that the range of $\tau$ can be widened to include those that
	satisfy $\tau\geq \frac{4}{1-\sqrt{1-\nu}}$. For instance, if
	we are given $\nu=3/4$, following \cref{eq:delta_h=0} we can
	choose $\tau = 12$. By \eqref{eq:alpha_range}, for this value we
	may allow any stepsize $\alpha$ that satisfies
	$3-\frac{\sqrt{21}}{3}<\alpha \Lf < 3+\frac{\sqrt{21}}{3}$, and the lower bound is strictly greater than 1. Nevertheless, we point out that this wider range of $\tau$ provided above is only sufficient to guarantee monotonicity of $\{V_k\}$. If we want to establish boundedness of a sequence generated by \cref{alg:main}, the restriction that $\alpha(\Lf+\Lh)<1$ will inevitably be required (see \cref{prop:bounded}).
 \end{remark}

 \subsection{Subsequential convergence}
 After obtaining the stepsize upper bounds, our next goal is to show
 subsequential convergence and convergence rates of \cref{alg:main}.
 We first establish the boundedness of its iterate sequence.

 \begin{proposition}\label{prop:bounded}
 	Suppose that \cref{assume:problem} holds, $\Psi$ has bounded level
sets and $\beta\in (0,\rhop^{-1}]$. %, and $\gamma\in (0,\rhog^{-1}]$. 
In
addition, suppose that $\Lf+\Lh>0$, and $\alpha, \tau>0$ are chosen such that
 	\begin{enumerate}[(a)]
 		\item $\tau \in (0,1]$ and $\alpha\in (0,\bar{\alpha})$, where $\bar{\alpha}$ is given in \cref{thm:Vk_nonincreasing_0to1}; 
 		\item $\tau\in (1,2)$ and $\alpha\in (0,\bar{\alpha})$, where $\bar{\alpha}$ is given in \cref{thm:Vk_nonincreasing_1to2}; or
 		\item   $\tau \geq 2$ satisfies \cref{eq:delta,eq:delta2}, and
			$\alpha\in \left(
			\frac{\tau\mu_*}{2(\Lf+\Lh)},\frac{\tau\mu^*}{2(\Lf+\Lh)}\right)\cap
			\left(0,\frac{1}{\Lf+\Lh}\right)$.
 	\end{enumerate}
 	If $\{(x^k,y^k,z^k)\}$ is generated by \cref{alg:main}, then 
 		\begin{enumerate}[(i)]
 			\item $\{(x^k,y^k,z^k)\}$ is bounded; and 
 			\item $\norm{(x^k,y^k,z^k) - (x^{k-1},y^{k-1},z^{k-1}) }\to 0$.
 		\end{enumerate}
 \end{proposition}
 \begin{proof}
	 We recall from \cref{eq:envelope} that
 $V_{\Lambda,\xi}(y,z,x) = \Phi_{\Lambda,\xi}(y^\plus; y,z,x)$ where
 $y^+ \in \argmin_{w\in\Re^n} \Phi_{\Lambda,\xi}(w; y,z,x)$. We then
 have from \cref{eq:Phi} that
 	\begin{align}
 		V_{\Lambda,\xi}(y,z,x) & = Q_{\alpha(f+h)}(y^\plus;x) + Q_{\beta p}(y^\plus;y,\xi) + g(y^\plus) \notag \\
 		& \geq f(y^\plus) + h(y^\plus) + \frac{1-\alpha (\Lf+\Lh)}{2\alpha} \norm{y^\plus - x}^2 +  Q_{\beta p}(y^\plus;y,\xi) + g(y^\plus) \notag \\
 		& \overset{\eqref{eq:Qpi_lowerbound}}{\geq}\Psi(y^+)  + \frac{1-\alpha (\Lf+\Lh)}{2\alpha} \norm{y^\plus - x}^2 + \frac{1-\beta \rhop}{2\beta} \norm{y^\plus - y}^2  ,\label{eq:Vlowerbound_Psi}
 	\end{align}
 	where the first inequality holds by $(\Lf+\Lh)$-smoothness of $f+h$ and \cref{lemma:properties_Lsmooth}(a). Setting
	$(y,z,x) = (y^{k-1},z^{k-1},x^{k-1})$, we have from
	\cref{eq:Vlowerbound_Psi} that
 	\begin{equation}
 		V_{k-1} \geq \Psi(y^k) + \frac{1-\alpha (\Lf+\Lh)}{2\alpha} \norm{y^k- x^{k-1}}^2 + \frac{1-\beta \rhop}{2\beta} \norm{y^k- y^{k-1}}^2 .
 		\label{eq:Vk_lowerbound}
 	\end{equation}
 	By our choice of $(\alpha,\beta,\gamma)$, it follows from
	\cref{thm:Vk_nonincreasing_0to1,thm:Vk_nonincreasing_1to2,thm:Vk_nonincreasing_>2}
	that $\{V_k\}$ is strictly decreasing.  Meanwhile, we have from
	\cref{remark:case2_magnitude} and \cref{remark:case2_magnitude_>1}
	that under conditions (a) and (b), respectively, it holds that $\alpha <
	\frac{1}{\Lf+\Lh}$. In condition (c), it is explicitly assumed
	that  $\alpha < \frac{1}{\Lf+\Lh}$. Hence, the second term in
	\eqref{eq:Vk_lowerbound} is always bounded below by zero. By our
	choice of $\beta$, the last term is also nonnegative. With these,
	we conclude that $\{\Psi(y^k)\}$ is bounded above, and therefore
	$\{y^k\}$ is a bounded sequence by the level-boundedness of
	$\Psi$. On the other hand, since $\Psi$ is closed, $\Psi$ is also
	bounded below. It thus
	follows from \cref{eq:Vk_lowerbound} that $\{\norm{y^k-x^{k-1}}\}$
	is bounded above. Since $\{y^k\}$ is bounded, it then follows that
	$\{x^{k-1}\}$ is also bounded. Finally, since $z^k = x^k+\alpha \nabla f(x^k)$ by \eqref{eq:x-step} and $\nabla f$ is continuous, we obtain the boundedness of $\{z^k\}$. This completes the proof of part (i). 
 	
 	To prove part (ii),  note that since $\{V_k\}$ is a bounded decreasing sequence, it follows that $V_{k-1}-V_k \to 0$.  Meanwhile, under our hypotheses, the inequality \eqref{eq:suff_decrease} holds and the coefficient of $\norm{x^k-x^{k-1}}^2$ is strictly positive since $\alpha < \frac{1}{\Lf+\Lh}$ as mentioned above. Using these facts, we have that $\norm{x^k-x^{k-1}}\to 0$. From \eqref{eq:z^k_change} and the $\Lf$-smoothness of $f$, we have
 	\begin{equation}
 	    \norm{z^k-z^{k-1}} = \norm{x^k - x^{k-1} + \alpha \left( \nabla f(x^k) - \nabla f(x^{k-1})\right) }\leq (1+\alpha \Lf)\norm{x^k-x^{k-1}},
      \label{eq:prox_lowerbound}
 	\end{equation}
 	and so $\norm{z^k-z^{k-1}}\to 0$. From \eqref{eq:z-step}, we have
	$\norm{y^{k}-x^{k-1}}\to 0$, and therefore $\norm{y^{k} -
	y^{k-1}}\to 0$ by the triangle inequality. 
 \end{proof}

 \begin{remark}
      When $\Lf+\Lh=0$, $x$ and $z$ play no role in the algorithm.  Similar to \eqref{eq:Vlowerbound_Psi}, we can obtain the following inequality:
   	\begin{align*}
 		V_{\Lambda,\xi}(y,z,x)\geq \Psi(y^+)  + \frac{1-\beta \rhop}{2\beta} \norm{y^\plus - y}^2  .
 	\end{align*}
       Hence, if $\Psi$ has bounded level sets and $\beta<\frac{1}{\rhop}$, we obtain the boundedness of $\{y^k\}$ and that $\norm{y^{k}-y^{k-1}}\to 0$. 
 \end{remark}

 We will next show that accumulation points of the $x$ and $y$
 sequence generated from \cref{alg:main} are \textit{stationary
 points} of $\Psi$. We say that a point $x^*$ is a {stationary point} of $\Psi$ (see \cite[Definition 4.1]{WenChenPong2018}) if 
  \begin{equation}
  0\in \nabla f(x^*) + \partial g(x^*) + \nabla h(x^*) + \partial p(x^*).
\label{eq:stationary}
  \end{equation}

Note that any local minimizer of $\Psi$ is a stationary point of $\Psi$ (see \cite[Theorem 2(i)]{PhamLeThi1997}). 
To establish stationarity of accumulation points, we will prove that
these points are fixed points of \alert{an operator associated with
\cref{alg:main}, which we derive as follows}. Observe that \cref{alg:main} can be concisely written
as fixed-point iterations of  \alert{the operator $T_{\Lambda}:\Re^n\times \Re^n \to \Re^n\times \Re^n$, given by
\[ (y^{k+1},z^{k+1}) \in T_{\Lambda}(y^k,z^k), \quad k=0,1,\dots, \]
where $T_{\Lambda}$ is defined as }
\begin{equation}
	T_{\Lambda}(y,z) \coloneqq  \left\lbrace \begin{pmatrix}
		y^\plus \\
		z + \tau (y^\plus - \prox_{\alpha f}(z))
	\end{pmatrix} : y^\plus  \in P_{\Lambda}(y,z) \right\rbrace .
 \label{eq:Tlambda}
\end{equation}
We say that $(y,z)$ is a \emph{fixed point} of $T_{\Lambda} $ if
$(y^*,z^*) \in T_{\Lambda}(y^*,z^*)$. The set of fixed points of
$T_{\Lambda}$ is denoted by $\Fix(T_{\Lambda})$.
The following proposition shows that fixed points of $T_\Lambda$ are
stationary points of $\Phi$.\cpsolved{I feel that we usually say it is not
difficult to prove when we want to omit the proof, so I removed that
sentence.}

\begin{proposition}
\label{prop:fixed_implies_stationary}
	Suppose that $f$ is $\Lf$-smooth, $g$ is proper and closed, $h$ is continuously differentiable, and $p$ is continuous on an open set containing $\dom (g)$. In addition, let $\alpha, \beta,
	\gamma>0$ satisfy $\frac{1}{\gamma}= \frac{1}{\alpha}+ \frac{1}{\beta}$ with $\alpha < \Lf^{-1}$. If $(y^*,z^*)\in \Fix (T_{\Lambda})$, then $y^*$ is a stationary point of $\Psi$. 
\end{proposition} 
\begin{proof}
	If $(y^*,z^*)\in \Fix (T_{\Lambda})$, then $y^*\in
	P_{\Lambda}(y^*,z^*)$, $z^* \in z^* + \tau (y^*-\prox_{\alpha f}(z^*))$.
	Hence, by noting that $\prox_{\alpha f}$ is single-valued since $\alpha<\Lf^{-1}$, we have 
	\begin{equation}
		y^* = \prox_{\alpha f}(z^*)
		\label{eq:ystar}
	\end{equation}
	and thus
	\begin{equation}
		 \frac{1}{\alpha}(z^*-y^*)= \nabla f(y^*).
		 \label{eq:from-z-update}
	 \end{equation} From \cref{eq:ystar} and the optimality condition
	 of \eqref{eq:Plambda}, we also have that there exists $\xi^*\in
	 \partial p(y^*)$ such that
	\begin{equation}
		 \frac{1}{\alpha} (2y^*-z^*-\alpha \nabla h(y^*))+\frac{1}{\beta}(y^*-\beta \xi^*) - \frac{1}{\gamma}y^* 
		\in \partial g (y^*).
		\label{eq:from-y-update}
	\end{equation}
	Adding \eqref{eq:from-z-update} and \eqref{eq:from-y-update} then leads to \cref{eq:stationary}.
 %$-\nabla h(y^*) - \xi^* \in \nabla f(y^*) + \partial g(y^*)$, proving the desired result.
\end{proof}

Thus, we define the \textit{residual function} for \cref{alg:main} as 
\begin{equation}
	R(y,z)\coloneqq \dist(0,(y,z)-T_{\Lambda}(y,z)),
	\label{eq:residual}
\end{equation}
and use it as a measure of stationarity from the view that $R(y,z)=0$ if and only if $(y,z)\in \Fix (T_{\Lambda})$.
This residual function reduces to that in 
\citep{LeeWright2019} when $f\equiv0$ and $p
\equiv 0$, and to that in \citep{LiuTakeda2022} when $f\equiv 0$. The next theorem shows that accumulation points of the $(y,z)$-sequence generated by \cref{alg:main} are fixed points of $T_{\Lambda}$.

\alert{
\begin{theorem}
\label{thm:accu_is_fixed}
	Suppose that the hypotheses of \cref{prop:bounded} hold, and let $(x^*,y^*,z^*)$ be an accumulation point of $\{(x^k,y^k,z^k)\}$. Then $x^*=y^*$ and $(y^*,z^*)\in \Fix (T_{\Lambda})$. 
\end{theorem}}
\begin{proof}
	Let $\{(x^{k_j},y^{k_j},z^{k_j})\}$ be such that
	$(x^{k_j},y^{k_j},z^{k_j}) \to (x^*,y^*,z^*)$. From the proof of
	\cref{prop:bounded}, we have that $\norm{y^k-x^{k-1}}\to 0$ and
	$\norm{y^{k}-y^{k-1}}\to 0$. It then follows that $y^{k_j}\to x^*$ by the triangle inequality, and therefore $x^*=y^*$. To show that $(y^*,z^*)\in \Fix (T_{\Lambda})$,  we first note that since $x^*=y^*$ and $z^k = x^k+\alpha \nabla f(x^k)$ by \eqref{eq:prox_optimality}, it follows that $z^*=y^*+\alpha \nabla f(y^*)$. Since $\alpha <\frac{1}{\Lf+\Lh}\leq \frac{1}{\Lf}$, then $f+\frac{1}{2\alpha}\|\cdot\|^2$ is strongly convex by \cref{remark:Lsmooth_implies_weaklyconvex} and therefore $y^*\in \prox_{\alpha f}(z^*)$.
    Thus, $z^* = z^* + \tau (y^* -\prox_{\alpha f}(z^*))$. Hence, to show that $(y^*,z^*)\in \Fix (T_{\Lambda})$, it remains to prove that $y^*\in P_{\Lambda}(y^*,z^*)$. To this end,  note that by the convexity and continuity of $\frac{\rhop}{2}\|\cdot\|^2 - p$ on $\dom(g)$
 (by \cref{assume:problem}(d)) and the formula \eqref{eq:partial_p}, we have from \cite[Theorem
 3.16]{Beck17} and \eqref{eq:partial_p} that $\{ \xi^{k_j} \}$ is bounded. Using again the continuity of $p$ on $\dom(g)$ together with \eqref{eq:generalsubdiff_consequence} and the fact that $\{ y^k\}\subseteq \dom (g)$, we see that accumulation points of $\{\xi^{k_j}\}$ belong to $\partial p(y^*)$. Without loss of generality, we may assume that $\xi^{k_j}\to \xi^*$ where $\xi^* \in \partial p(y^*) $.

Meanwhile,  we have from  \eqref{eq:y-step} that  
\begin{equation}
    y^{k_j+1} \in \prox_{\gamma g}( u^{k_j}) \quad \text{where}~u^{k_j} \coloneqq  \frac{\gamma}{\alpha} (2x^{k_j}-z^{k_j} - \alpha \nabla h(x^{k_j}) + \frac{\gamma}{\beta}(y^{k_j} - \beta \xi^{k_j}).
    \label{eq:ykj+1}
\end{equation}
Note that $y^{k_j+1}\to y^*$ and $u^{k_j}\to u^* \coloneqq \frac{\gamma}{\alpha}(2y^*-z^* - \alpha \nabla h(y^*)) + \frac{\gamma}{\beta}(y^* - \beta \xi^*)$, where we have used the fact that $\xi^{k_j}\to \xi^*$ and $x^*=y^*$. To prove the claim that $y^*\in P_{\Lambda}(y^*,z^*)$, we only need to show that $y^*\in \prox_{\gamma g}(u^*)$, noting that $y^*\in \prox_{\alpha f}(z^*)$   .

From \cref{lemma:properties_Lsmooth}(a) and \eqref{eq:p_descentlemma}, we have that for all $x\in \Re^n$, 
\begin{align*}
    \Psi^* & \leq f(x) + h(x) + p(x) + g(x) \\
    & \leq (f+h+p)(\bar{x}) + \lla \nabla (f+h)(\bar{x}) + \bar{\xi} , x-\bar{x} \rla +  \frac{\Lf + \Lh +\rhop}{2}\norm{x-\bar{x}}^2 + g(x), 
\end{align*}
where $\bar{x}$ is a fixed but arbitrary element of $\dom (g)$ and $\bar{\xi}\in \partial p(\bar{x})$. We may write the above inequality as
\begin{equation}
    0 \leq a + \lla v, x\rla + \frac{\Lf+\Lh+\rhop}{2}\norm{x}^2 + g(x) 
    \label{eq:prox_bounded}
\end{equation}
for some $a\in \Re$ and $v\in \Re^n$. It then follows that $\liminf_{\norm{x}\to\infty} \frac{g(x)}{\norm{x}^2} \geq -\frac{\Lf+\Lh+\rhop}{2}$ by dividing both sides of \eqref{eq:prox_bounded} by $\norm{x}^2$ and using Cauchy-Schwarz inequality. By \cite[Exercise 1.24]{RW98}, %it follows that $g$ is prox-bounded with threshold at least $(\Lf+\Lh+\rhop)^{-1}$
\jhsolved{Will probably need to add the definition of this later, if you agree with my proof}\cpsolved{Tried to workaround the term. How about the current description?}\jhsolved{It looks good! Thanks for this!} the fact that $\frac{1}{\gamma} > \Lf+\Lh + \frac{1}{\beta}\geq \Lf+\Lh+\rhop$ (that is, $\gamma <\frac{1}{\Lf+\Lh+\rhop}$), and \cite[Theorem 1.25]{RW98}, we know that $\{y^{k_j+1}\}$ given in \eqref{eq:ykj+1} is bounded,\jhsolved{This theorem, together with the proof of prox-boundedness, actually implies that your question before about the boundedness of $\prox_{\gamma g}(\cdot)$ indeed holds.}\cpsolved{I see it is the first part of the thm that not used here. Thanks!} with accumulation points lying in $\prox_{\gamma g}(u^*)$. Since $y^{k_j+1} \to y^*$, it follows that $y^* \in \prox_{\gamma g}(u^*)$, as desired.
\end{proof}

\alert{
\begin{remark}
	\label{rem:stationary}
 It is not difficult to verify that when $\alpha < \Lf^{-1}$, $(y^*,z^*)\in \Fix (T_{\Lambda})$ is equivalent to the conditions
 \begin{equation}
     \begin{cases}
         & y^* \in \prox_{\gamma g}\left( \frac{\gamma}{\alpha}\left(
		 y^*-\alpha \nabla f(y^*) - \alpha \nabla h(y^*)\right) +
		 \frac{\gamma}{\beta}T_{\beta p}(y^*)\right), \\
         & z^* = y^*+\alpha \nabla f(y^*).
     \end{cases}   
     \label{eq:stopcriterion_basis}
 \end{equation}
Meanwhile, in certain applications \cite{LuZ19a,LuZS19a}, it is of
interest to consider cases in which $p$ takes the form  
\begin{equation}
p(x) = - \max_{y\in \mathcal{Y}} \phi(x,y),
\end{equation}  
where $\phi(\cdot,y)$ is continuously differentiable and convex for
all $y \in \mathcal{Y}$, and $\mathcal{Y}$ is a compact set. This
satisfies \cref{assume:problem}(d) with $\rhop = 0$. If $p$ takes this
form \cpmodifyok{and $f$, $h$, $g$ are convex and satisfy}
\cref{assume:problem}(a)-(c), then from \cite[Proposition~1]{LuZS19a}
and \eqref{eq:stopcriterion_basis}, it follows that any point
$(y^*,z^*) \in \Fix (T_{\Lambda})$ ensures that $y^*$ corresponds to a
D-stationary point \cpmodifyok{of \cref{eq:problem}}, which\cpmodifyok{, as
	discussed in \cite{PangRA17a},} is generally
a stronger notion than the \cpmodifyok{stationarity we used in
	\cref{eq:stationary}}.
Together with \cref{thm:accu_is_fixed}, this shows that under
\cpmodifyok{these} conditions on $f, g, h$, and $p$, our proposed algorithm is globally subsequentially convergent to a point $(y^*,z^*)$ where $y^*$ is a D-stationary point of \eqref{eq:problem}.  
\end{remark}
}

% A direct consequence of inequality \eqref{eq:suff_decrease} and \cref{prop:bounded} is the iteration complexity given in the following result.

The following theorem provides the $O(1/T)$ iteration complexity of \cref{alg:main}.
\alert{
\begin{theorem}
\label{thm:complexity}
	Suppose that the hypotheses of \cref{prop:bounded} hold. There
	exists $\omega>0$ such that $		\min _{k=1,2,\dots, T }
	R(y^k,z^k)^2 \leq \frac{\omega}{T}, $ for all $T>1$.\cpsolved{Changed
	the counter to $T$, which seems to be more widely used.}
	% \begin{equation*}
	% 	\min _{k=1,2,\dots, N } R(y^k,z^k)^2 \leq \frac{\omega}{N}, 
	% \end{equation*}
	% where $\omega>0$ is given by 
 %        \ifdefined\submit 
	% \begin{equation*}
	% 	\omega \coloneqq \begin{cases}
	% 		\frac{2\beta (V_0-\Psi^*)}{\min \left\lbrace 2\beta \zeta, 1-\beta\rhop \right\rbrace} & \text{if}~\beta<\rhop^{-1}, \\
	% 		\frac{4\gamma (V_0-\Psi^*)}{\min \left\lbrace 2\gamma\zeta, 1-\gamma\rhog\right\rbrace} & \text{if}~\gamma < \rhog^{-1},
	% 	\end{cases} \quad \zeta \coloneqq \frac{-c(\alpha)}{2\tau\alpha(1+\alpha \Lf)^2}>0, \quad \Psi^* \coloneqq \min_{w\in \Re^n} \Psi(w).
	% \end{equation*}
 %        \else
 %        \begin{equation*}
	% 	\omega \coloneqq \begin{cases}
	% 		\frac{2\beta (V_0-\Psi^*)}{\min \left\lbrace 2\beta \zeta, 1-\beta\rhop \right\rbrace} & \text{if}~\beta<\rhop^{-1}, \\
	% 		\frac{4\gamma (V_0-\Psi^*)}{\min \left\lbrace 2\gamma\zeta, 1-\gamma\rhog\right\rbrace} & \text{if}~\gamma < \rhog^{-1},
	% 	\end{cases} \quad \text{and} \quad \zeta \coloneqq \frac{-c(\alpha)}{2\tau\alpha(1+\alpha \Lf)^2}>0, 
	% \end{equation*}
	% and $\Psi^* \coloneqq \min_{w\in \Re^n} \Psi(w)$. 
    % \fi
\end{theorem}
% \begin{proof}
% 	To prove the last claim, we note first from \eqref{eq:Vk_lowerbound} that $V_k \geq \Psi^*$. Summing \eqref{eq:suff_decrease} from 2 to $N+1$, we get
% 	\begin{equation}
% 		\sum_{k=2}^{N+1} -\frac{c(\alpha)}{2\tau \alpha}\norm{x^k-x^{k-1}}^2 + \sum_{k=2}^{N+1}\frac{1-\beta\rhop}{2\beta}\norm{y^k-y^{k-1}}^2  \leq V_1-\Psi^*.
%     \label{eq:2toN+1}
% 	\end{equation}
%  Meanwhile, we have $R(y^k,z^k) ^2\leq \norm{(y^k,z^k)-(y^{k+1},z^{k+1})}^2$ since $(y^{k+1},z^{k+1}) \in T_{\Lambda}(y^k,z^k)$. These together with \eqref{eq:suff_decrease} and \eqref{eq:prox_lowerbound} give the desired result for the case $\beta<\rhop^{-1}$. For the other case, the summation of \eqref{eq:suff_decrease} from 1 to $N$ gives
% 	\begin{equation*}
% 		\sum_{k=1}^N\frac{1-\gamma\rhog}{2\gamma}\norm{y^k-y^{k+1}}^2 \leq V_0-\Psi^*.
% 	\end{equation*}
% Using \eqref{eq:2toN+1}, we obtain 
%     \begin{equation*}
%         \sum_{k=1}^{N} -\frac{c(\alpha)}{2\tau \alpha}\norm{x^{k+1}-x^{k}}^2+\sum_{k=1}^N\frac{1-\gamma\rhog}{2\gamma}\norm{y^k-y^{k+1}}^2 \leq  V_0+V_1-2\Psi^*,
%     \end{equation*}
% 	from where we can easily infer the result.
% \end{proof}
\begin{proof}
From \eqref{eq:suff_decrease}, we have for any $N\geq 0$ that
      \begin{equation}
        \sum_{k=1}^{N+1} -\frac{c(\alpha)}{2\tau \alpha}\norm{x^{k}-x^{k-1}}^2 \leq  V_0-\Psi^*,
        \label{eq:fromsuffdecrease_noweakconvexity}
    \end{equation}
    where  $\Psi^* \coloneqq \min_{x\in \Re^n} \Psi(x)$.
	%We note that $c(\alpha)<0$ by the choice of $\alpha$ in \cref{prop:bounded}.
     On the other hand, from \cref{eq:z-step}, we get $y^{k+1} - y^k = x^{k} - x^{k-1} +
	 \tau^{-1} (z^{k+1} - z^k) - \tau^{-1} (z^k - z^{k-1}).$
    \ifdefined\submit 
	Using this fact, the triangle inequality,  the estimate
	$(a+b+c)^2 \leq 3(a^2+b^2+c^2)$ (for $a,b,c\in\Re$), that
	$c(\alpha) < 0$ by the choice of $\alpha$ in \cref{prop:bounded},
and \eqref{eq:fromsuffdecrease_noweakconvexity}, we  get the desired result. 
 \else 
  It follows from the triangle inequality and the estimate $(a+b+c)^2 \leq 3(a^2+b^2+c^2)$ (for $a,b,c\in\Re$) that
  \begin{equation*}
    \norm{y^{k+1} - y^k}^2 \leq 3\norm{x^{k} - x^{k-1}}^2 +
	 3\tau^{-2} \norm{z^{k+1} - z^k}^2 + 3\tau^{-2} \norm{z^k - z^{k-1}}^2.
  \end{equation*}
  Then 
    \begin{align*}
    \sum_{k=1}^{N} \norm{y^{k+1} - y^k}^2 & \leq 3\sum_{k=1}^{N}  \norm{x^{k} - x^{k-1}}^2 +
	 3\tau^{-2} \sum_{k=1}^{N}  \norm{z^{k+1} - z^k}^2 + 3\tau^{-2} \sum_{k=1}^{N}  \norm{z^k - z^{k-1}}^2 \\
  & \overset{\eqref{eq:prox_lowerbound}}{\leq } 3\sum_{k=1}^{N}  \norm{x^{k} - x^{k-1}}^2 +
	  3\tau^{-2} (1+\alpha \Lf)^2 \sum_{k=1}^{N} \norm{x^{k+1}-x^{k}}^2 \\
   &\quad +   3\tau^{-2} (1+\alpha \Lf)^2 \sum_{k=1}^{N} \norm{x^{k}-x^{k-1}}^2 \\
   & = \left( 3+3\tau^{-2}(1+\alpha\Lf)^2\right) \sum_{k=1}^{N}  \norm{x^{k} - x^{k-1}}^2 +3\tau^{-2}(1+\alpha\Lf)^2  \sum_{k=1}^{N}\norm{x^{k+1}-x^{k}}^2 \\
   & \overset{\eqref{eq:fromsuffdecrease_noweakconvexity}}{\leq}  \left( 3+6\tau^{-2}(1+\alpha\Lf)^2\right) \left( -\frac{2\tau\alpha (V_0-\Psi^*)}{c(\alpha)}\right).
  \end{align*}
  Combining this with \eqref{eq:fromsuffdecrease_noweakconvexity} yields
  \begin{equation*}
  \label{eq:xy}
    \sum_{k=1}^{N} \norm{x^{k+1} - x^k}^2 + \sum_{k=1}^{N} \norm{y^{k+1} - y^k}^2  \leq \left( 4+6\tau^{-2}(1+\alpha\Lf)^2\right) \left( -\frac{2\tau\alpha (V_0-\Psi^*)}{c(\alpha)}\right) ,
  \end{equation*}
  from where the result follows.
  \fi 
\end{proof}
}

While finalizing this paper, we became aware of a recent preprint by
\citet{Dao2024} that proposed to solve \cref{eq:problem} with $\rhop=0$
using an algorithm similar to ours.\footnote{Specifically, setting $\theta=1$ in \cite[Algorithm 1]{Dao2024} coincides with our algorithm for the case of $\rhop=0$.}
The main advantage of our analysis for this coincided special
case is that our upper bounds for
stepsizes are significantly larger than those in
\citep{Dao2024}; see
\cref{fig:compare_stepsize2,fig:compare_stepsize3} for geometric
illustrations.
\alert{Numerical results in \cref{sec:stepsize} indicate that such a larger
$\alpha$ is meaningful as it provides faster practical convergence.}
Meanwhile, subsequential convergence to stationary points is proven in both works, but ours allows for a wider range of stepsizes. \citep{Dao2024} also proved global convergence of the full sequence under a Kurdyka-\L ojasiewicz hypothesis, wherein they used the techniques by \citep{liupongtakeda2019}.
We note that a similar strategy can be employed to establish the same global convergence property for our algorithm with a wider range of stepsizes.

\ifdefined\submit
  \def\figscaledao{0.35}
\else
  \def\figscaledao{.55}
\fi

\begin{figure}[tb!]
	\centering
	\begin{tabular}{@{}cc@{}}
        \includegraphics[scale=\figscaledao]{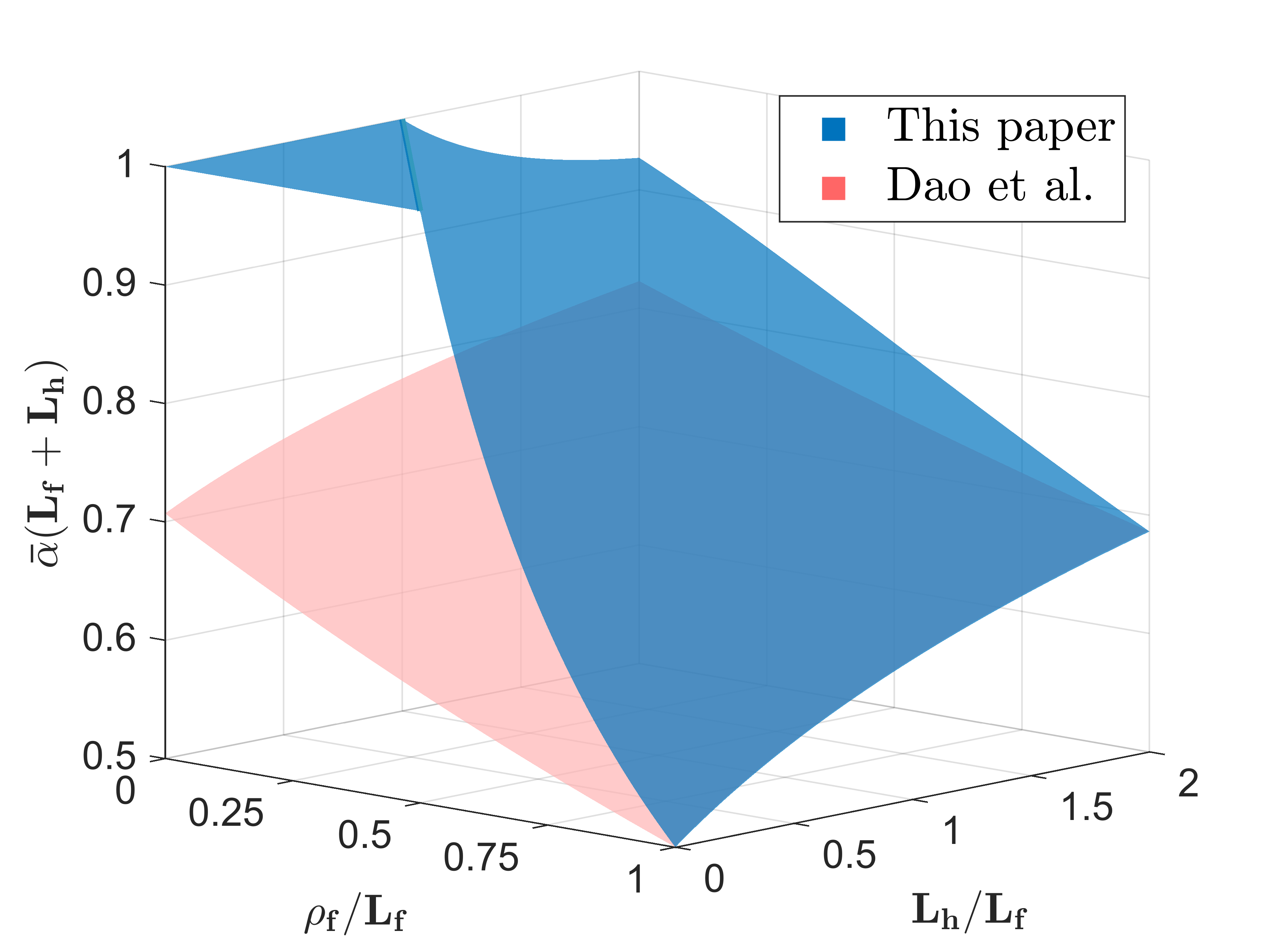}&
		\includegraphics[scale=\figscaledao]{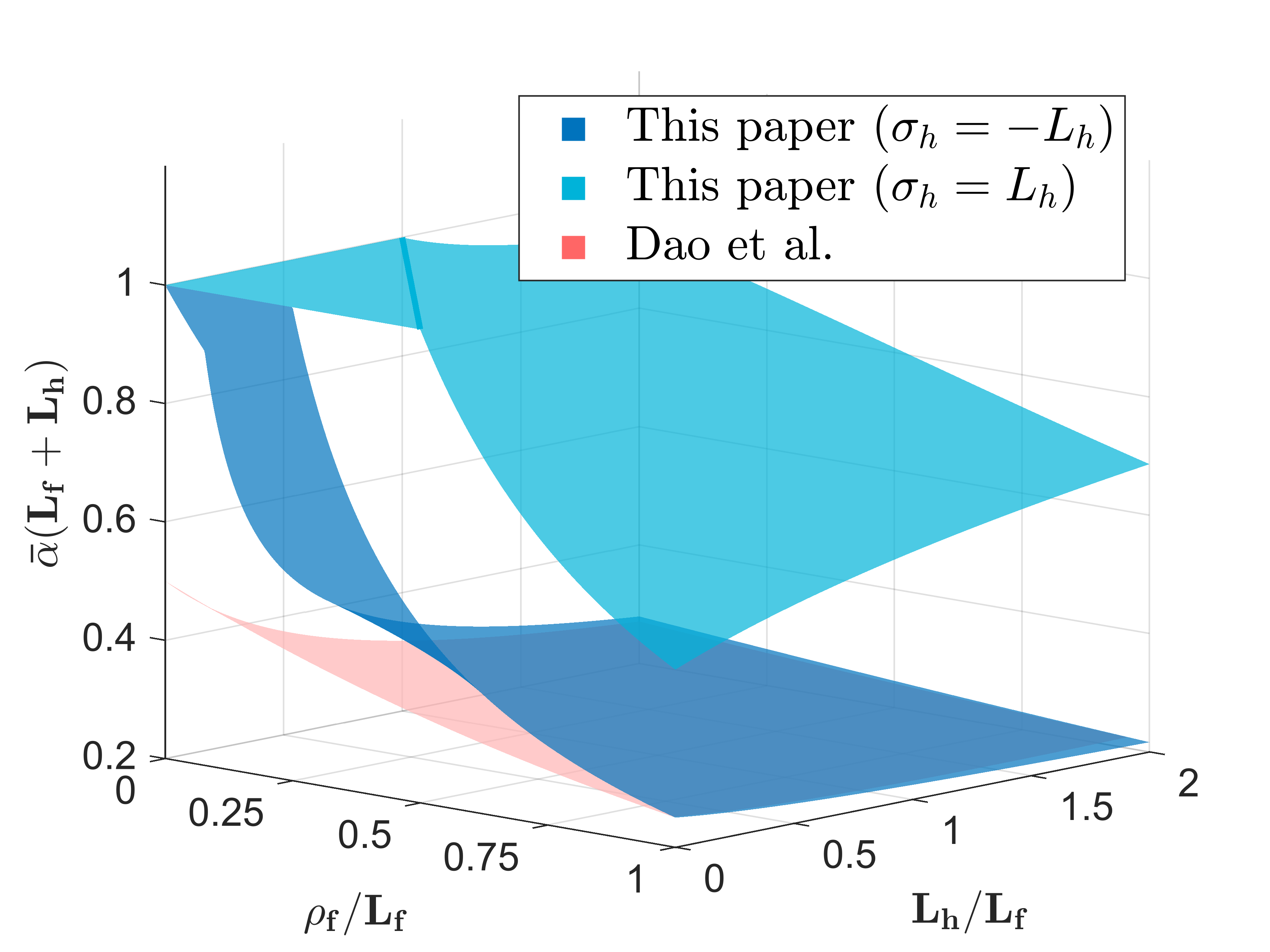}\\
		(a) $\tau=1$ & (b) $\tau=1.5$ and $\sigmah = \pm\Lh$
	\end{tabular}
	\caption{{\small Comparison of stepsize upper bounds (denoted by
		$\bar{\alpha}$) computed in the present work and in
		\citep{Dao2024} for $\tau = 1$ and $\tau=1.5$. We can see that
		the stepsize upper bounds in this work are always larger
		than those in \citep{Dao2024}. Figure (a)
		demonstrates that the upper bound we have derived for the
		stepsize of the Davis-Yin splitting algorithm is
		$\frac{1}{\Lf+\Lh}$ provided that
		$\frac{\Lh}{\Lf}+\frac{2\rhof}{\Lf}\leq 1$ (see
		\cref{thm:Vk_nonincreasing_0to1}). Figure (b) demonstrates
		\cref{thm:Vk_nonincreasing_1to2} with $\tau=1.5$ using a rough
		estimate $-\sigmah = \Lh$, but this can be further
		improved when $-\sigmah < \Lh$. For instance, the case of
		$h(x) = \lambda \|x\|_2^2$ $(\sigmah = \Lh)$ is also shown in
		figure (b) (see \cref{remark:tau>1_samebounds}).}}
	\label{fig:compare_stepsize2}
\end{figure}

\begin{figure}[tb!]
    \centering
     \includegraphics[scale=\figscaledao]{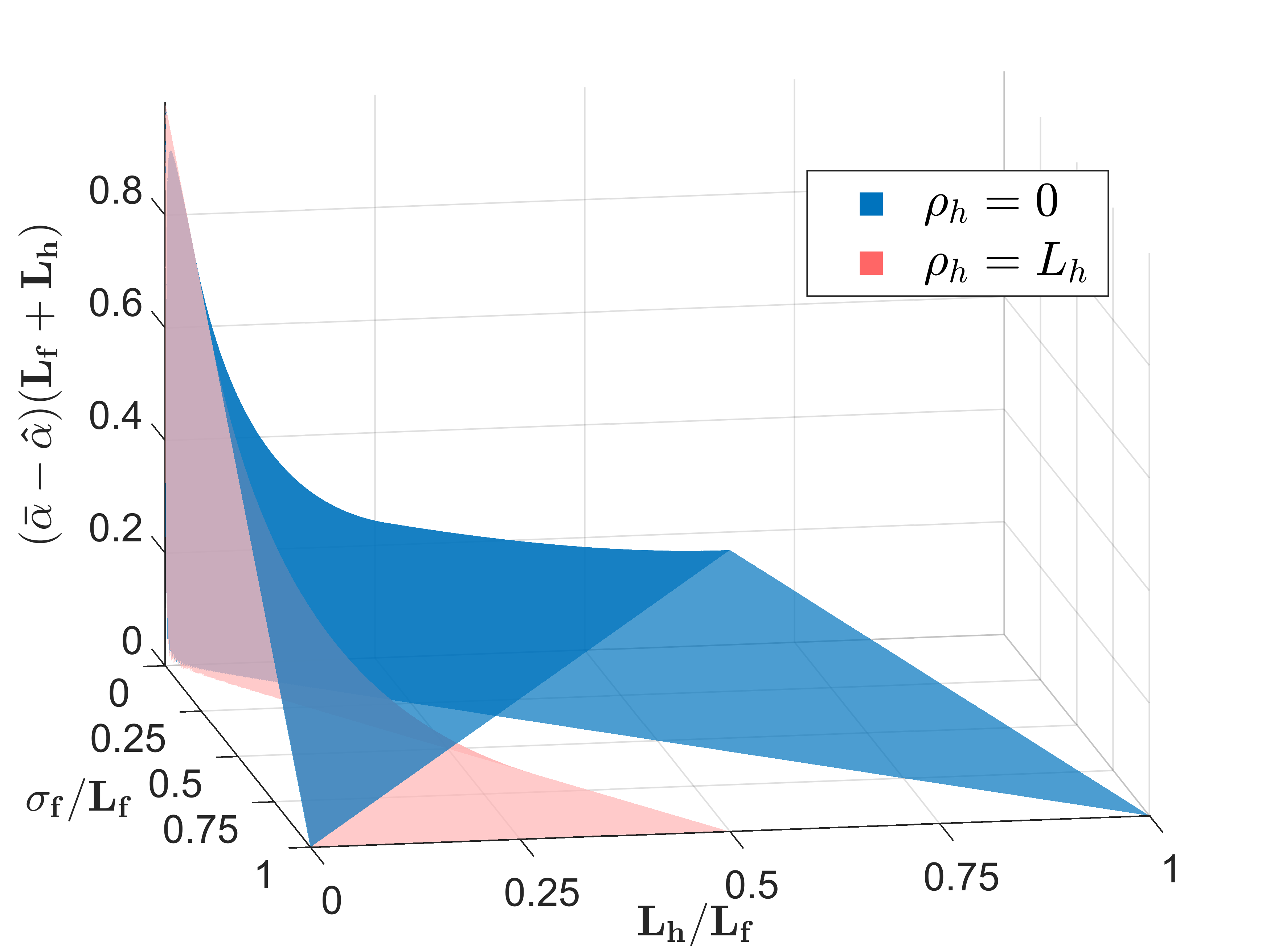}
    \caption{{\small Gaps between the stepsize upper bounds in the present work and in \citep{Dao2024}, denoted respectively by $\bar{\alpha}$ and $\hat{\alpha}$, when $\tau = 2$. In this case, \citep{Dao2024} always requires that $\sigmaf>2\Lh$. Using the rough estimate $\rhoh=\Lh$ in \cref{thm:Vk_nonincreasing_1to2}, the red graph shows that the difference $\bar{\alpha}-\hat{\alpha}$ between our and their stepsizes is always nonnegative. For a convex $h$ (namely when $\rhoh = 0$), the upper bound $\hat{\alpha}$ provided in \citep{Dao2024} does not change. As depicted in the blue graph, ours, on the other hand, provides a larger stepsize upper bound for the convex case, resulting in a bigger gap between $\bar{\alpha}$ and $\hat{\alpha}$. Ours also has a wider region in which $\tau=2$ is applicable, since our condition on the parameters (namely $\sigmaf > \Lh +\rhoh$) becomes weaker than requiring $
    \sigmaf > 2\Lh$.}}
    \label{fig:compare_stepsize3}
\end{figure}

\section{Numerical experiments}
\label{sec:experiments}
In this section, we present experiments to corroborate the
theoretical findings in the previous section.
All experiments were conducted in MATLAB. 
%running on a 64-bit machine
%with Intel Xeon Silver 4208 CPU and 128GB memory.
For our method, we experiment with various values of $\tau$.
\alert{For each $\tau$, we use the corresponding $\bar{\alpha}$ in
\cref{thm:Vk_nonincreasing_0to1,thm:Vk_nonincreasing_1to2,thm:Vk_nonincreasing_>2}, and we calculate $\alpha$ by multiplying $\bar{\alpha}$ by $0.99$ to ensure global subsequential convergence as suggested
by \cref{thm:accu_is_fixed}.
We also set $\beta = 1/\rhop$ according to our theoretical results when
the $p$ term is present, and set it to \cpmodifyok{$\infty$} otherwise
\cpmodifyok{as described in \cref{subsec:alg}}.
}
\cpmodifyok{The value of $\gamma$ is then decided by $\alpha$ and
	$\beta$ through Step~0 of \cref{alg:main}.}

\alert{For a fair comparison of different methods that can fit in the framework of
	\cref{alg:main}, \cpmodifyok{instead of \cref{eq:residual},} we
	report their \cpmodifyok{respective} required number of iterations and
running time to reach \cpmodifyok{an approximate first-order optimality measure
independent of the algorithm}. \cpmodifyok{In particular,} in view of \eqref{eq:stopcriterion_basis},
\cpmodifyok{for any algorithm compared, regardless
of their specific reformulation/transformation, we stop it 
when its main iterate $y^k$ satisfies}
    \begin{equation}
         \dist \left( 0, y^k - \prox_{\hat{\gamma} g}\left(
		 \frac{\hat{\gamma}}{\hat{\alpha}}\left( y^k-\hat{\alpha}
		 \nabla f(y^k) - \hat{\alpha}\nabla h(y^k)\right) +
		 \frac{\hat{\gamma}}{\hat{\beta}}T_{\hat{\beta} p}(y^k)\right)\right)\leq \epsilon
		 \label{eq:stationary_measure}
    \end{equation}
	\cpmodifyok{ for some $\epsilon > 0$, where the choice of $f,g,h$,
		and $p$ are those used in our algorithm, and
		$(\hat{\alpha},\hat{\beta},\hat \gamma) =
	(\alpha,\beta,\gamma)$ are calculated following the above
procedure for our algorithm with $\tau = 1$.
From \cref{rem:stationary}, it is clear that at any stationary point
$y^*$ of \cref{eq:problem}, the left-hand side of
\cref{eq:stationary_measure} vanishes, and thus it is a valid measure
for the approximate first-order optimality measure.
In all experiments in this section, we set $\epsilon = 10^{-6}$.}}

\subsection{Nonnegative and low-rank matrix completion}
The first experiment we consider is nonnegative matrix
completion/factorization \citep{lee1999learning}, whose goal is to
recover missing entries of a partially observed matrix using
nonnegative entries, and we impose the additional structure requirement
that the recovered matrix should be low-rank.
Given a data matrix $M \in \Re^{m\times n}$, we let $\Omega$ denote the
entries $(i,j)$ \alert{observed} and $P_{\Omega}$ denote the
associated projection operation such that
$P_{\Omega}(X)_{i,j}$ outputs $X_{i,j}$ if $(i,j)\in \Omega$ and $0$ otherwise.
We consider \cref{eq:problem} with
$p(X) \equiv 0$ and
\begin{equation}
	f(X) = \frac{\lambda_1}{2} \min_{Y \in \Re^{m\times
	n}_+}\|X-Y\|^2_F, \quad g(X) = \lambda_2
	\|X\|_*, \quad h(X) = \frac12 \|P_{\Omega}(X - M)\|^2_F,
	\label{eq:nmf}
\end{equation}
where $\lambda_1,\lambda_2 \geq 0$ are weights for the respective terms,
$\|\cdot\|_*$ is the nuclear norm, $\|\cdot\|_F$ is the Frobenius
norm, and $\Re^{m \times n}_+$ is the nonnegative orthant of $\Re^{m
\times n}$.
\alert{We can easily see that \cref{eq:nmf} satisfies
our assumptions with $L_f = \lambda_1,
L_h = 1$, and $\rho_p = \rho_f = \rho_h = \sigma_h = 0$.}

We follow \cite{TohYun10a} to generate $M \in \Re^{m\times n}$ with a specific
rank $r$ as the product of an $m\times r$ and an $r\times n$ matrix
whose entries are all identically and independently distributed as
standard Gaussian, and then randomly select $s$ entries in uniform
random to form $\Omega$.
We in particular use $m=n$ and consider $(n,s) \in \{(100,1000),
(300,10000)\}$, $r \in \{10,30\}$, $\lambda_1 = 10$ and $\lambda_2 = 5$.
We set $\beta = \infty$ for all applicable methods due to the absence
of the $p$ term.
All algorithms are run with a cap of $30,000$ iterations.

\alert{\subsubsection{Comparison of different bounds for $\alpha$}
\label{sec:stepsize}
We first use our method with $\tau = 1$ to recover the DYS algorithm,
and compare our obtained bound for $\bar \alpha$ with those in
\cite{BianZhang21,Dao2024} for the same setting.
We use the same factor of $0.99$ for all the obtained upper bounds to
ensure that strict inequality holds, which is essential to guarantee theoretical convergence.
The result in \cref{tbl:nmfDYS} shows that a larger $\alpha$ leads to
faster convergence, and therefore our new upper bound for $\alpha$
provides not only theoretical but also practical improvements.
}
\begin{table}[tb]
	\alert{
	\centering
    	\caption{Comparison between different bounds for $\alpha$ in
			DYS using nonnegative and low-rank
	matrix completion. $*$ denotes that the maximum iteration of
$30000$ is reached.
The respective values of $\bar \alpha$ in this experiment are:
\cite{BianZhang21}: $0.0212$, \cite{Dao2024}: $0.0645$, ours: $0.09$.
}
	\label{tbl:nmfDYS}
	\begin{tabular}{@{}l|rr|rr|rr|rr@{}}
		\multirow{2}{*}{Bound} & \multicolumn{2}{c|}{$n=100,r=10$} &
		\multicolumn{2}{c|}{$n=100,r=30$} &
		\multicolumn{2}{c|}{$n=300,r=10$} &
	\multicolumn{2}{c}{$n=300,r=30$}\\
	\cmidrule{2-9}
& iter & time& iter & time& iter & time& iter & time\\
	\midrule
\cite{BianZhang21} & 14371 & 17.37 & 27335 & 32.64 & 14937 & 154.11 & \conditionalstar{*30000} & \conditionalstar{*322.91} \\
\cite{Dao2024} & 9087 & 11.89 & 17283 & 21.70 & 9445 & 100.31 & 21929 & 229.87 \\
Ours  & 6516 & 8.71 & 12394 & 16.47 & 6773 & 73.02 & 15725 & 167.23
	\end{tabular}
}
\end{table}

\subsubsection{Comparison with other methods}
We next compare our method with different $\tau$'s.
Since $f$ is not strongly convex, our theoretical
results indicate that we should use $\tau < 2$.
We thus run our method
with $\tau \in
\{1,1.1,\cdots,1.9\}.\footnote{We have
	also experimented with $\tau\in \{0.5,0.6,0.7,0.8,0.9\}$, but those values
give worse performance, with smaller $\tau$ leading to slower convergence.
We therefore omit these results in our experiments.}$
% \cp{I am leaving it here right now for
% 	your reference, but it still seems to me that this part is not
% very meaningful to be presented. The part of interchanging $f$ and $h$
% should also be removed.}\jh{For me, this provides a little value too. But we can probably include it for this experiment (Section 4.1) just to illustrate that it works. And then in Section 4.2, we can omit the experiments for $\tau<1$. We can just mention that $\tau <1$, similar to the preceding experiment, resulted to worse performance (similar to the approach we took in our first version) and is thus omitted. What do you think? }

%\cpsolved{I feel we can actually remove results with $\tau < 1$ (in all exps). What
%do you guys think?}\jhsolved{I think it's okay to remove them; I have no strong opinions on this.}
%\memo{AT}{I think it's okay, too!}
\alert{We compare our method with the proximal gradient method (PG) by treating $f+h$
together as the smooth term, and we set the stepsize \cpmodifyok{to be}
$0.99/(\Lf+\Lh)$ \cpmodifyok{following standard practice}}.
In addition, as noticed in the previous part, for this problem, our method with $\tau = 1$ recovers the
Davis-Yin splitting (DYS).
According to the result of the previous experiment, for DYS we use the
new upper bound for $\alpha$ obtained in this work.
We can clearly see from the results in \Cref{tbl:nmf} that our
method with $\tau \in (1,1.9)$ constantly outperforms DYS (with our improved upper bound for
$\alpha$) and PG, and DYS and PG are performing very similarly in this
experiment.

\begin{table}
	\centering
    	\caption{Comparison between methods for nonnegative and low-rank
	matrix completion. $*$ denotes that the maximum iteration of $30000$ is reached.
$\tau = t$ denotes our method with the specified $\tau$.}
	\label{tbl:nmf}
	\begin{tabular}{@{}l|rr|rr|rr|rr@{}}
		\multirow{2}{*}{Method} & \multicolumn{2}{c|}{$n=100,r=10$} &
		\multicolumn{2}{c|}{$n=100,r=30$} &
		\multicolumn{2}{c|}{$n=300,r=10$} &
	\multicolumn{2}{c}{$n=300,r=30$}\\
	\cmidrule{2-9}
& iter & time& iter & time& iter & time& iter & time\\
	\midrule
	PG	 & 6516 & 8.48 & 12393 & 16.45 & 6773 & 72.24 & 15725 & 168.03 \\
% EPDCA	 & 976 & 1.35 & 1890 & 2.64 & 1087 & 11.37 & 2473 & 26.49 \\
% $\tau = 0.5$ & 13036 & 16.14 & 24791 & 29.89 & 13551 & 142.76 & \conditionalstar{*30000} & \conditionalstar{*314.96} \\
% $\tau = 0.6$ & 10863 & 14.14 & 20658 & 28.08 & 11292 & 118.06 & 26212 & 273.34 \\
% $\tau = 0.7$ & 9310 & 12.54 & 17707 & 23.04 & 9678 & 101.14 & 22467 & 234.53 \\
% $\tau = 0.8$ & 8146 & 10.87 & 15493 & 20.26 & 8468 & 88.65 & 19658 & 205.82 \\
% $\tau = 0.9$ & 7240 & 9.64 & 13771 & 17.88 & 7526 & 78.49 & 17473 & 183.68 \\
DYS & 6516 & 8.71 & 12394 & 16.47 & 6773 & 73.02 & 15725 & 167.23 \\
$\tau = 1.1$ & 5963 & 8.03 & 11342 & 14.61 & 6198 & 65.64 & 14391 & 152.51 \\
$\tau = 1.2$ & 5505 & 7.33 & 10472 & 14.10 & 5723 & 60.82 & 13287 & 142.07 \\
$\tau = 1.3$ & 5122 & 6.47 & 9742 & 13.16 & 5324 & 56.25 & 12362 & 136.20 \\
$\tau = 1.4$ & 4796 & 6.07 & 9124 & 12.02 & 4985 & 52.89 & 11577 & 123.30 \\
$\tau = 1.5$ & 4518 & 5.76 & 8595 & 11.35 & 4697 & 49.69 & 10906 & 116.72 \\
$\tau = 1.6$ & 4279 & 5.47 & 8141 & 10.74 & 4448 & 47.04 & 10330 & 109.19 \\
$\tau = 1.7$ & \bftab 4074 &\bftab  5.28 &\bftab  7750 &\bftab  10.22
&\bftab  4235 &\bftab  44.96 &\bftab  9834 &\bftab  104.33 \\
$\tau = 1.8$ & 5087 & 6.57 & 9678 & 12.78 & 5288 & 56.09 & 12280 & 131.24 \\
$\tau = 1.9$ & 8270 & 10.70 & 15732 & 20.78 & 8595 & 91.02 & 19961 & 213.87 \\
	\end{tabular}
\end{table}

\subsection{Cardinality-constrained problems}
Our next experiment considers a \alert{least-squares} problem with a penalized form of the cardinality
constraint $\|x\|_0 \leq k$ for some given positive integer $k$.
As argued in \citep{GotohTakeadaTono18a}, this constraint is equivalent to
	$\|x\|_1 - \|x\|_{(k)} = 0$,	
where $\|x\|_{(k)}$ is the Ky Fan $k$-norm that computes the sum of
the top $k$ largest elements of the element-wise absolute value of $x$,
and we follow \cite{GotohTakeadaTono18a} to take $\|x\|_1 -
\|x\|_{(k)}$ as a penalty in the objective function instead of
enforcing the hard constraint.
We also follow a common convention in machine learning to introduce a Tikhonov regularization of the form $\|x\|_2^2$ to improve the problem condition.
This leads to the following setting for \cref{eq:problem}.
\begin{equation}
	f(x) = \frac12 \|Ax - b\|^2_2,
	\quad g(x) = \lambda_2
	\|x\|_1,\quad p(x) = -\lambda_2 \|x\|_{(k)}, \quad
	h(x) = \frac{\lambda_1}{2}\|x\|_2^2.
	\label{eq:kyfan}
\end{equation}
%For $h$, we consider both the least-square loss for regression and the
%logistic loss for binary classification.
%Given a data matrix $A = (a_1^{\top},\dotsc,a_m^{\top})\in \Re^{m \times n}$ and a response vector $b
%\in \Re^{m}$, the former has the form
%\begin{equation}
%	h(x) = \frac12 \|Ax - b\|^2_2,
%	\label{eq:ls}
%\end{equation}
%and the latter has the form
%\begin{equation}
%	h(x) = \sum_{i=1}^m \log(1 + \exp(-b_i a_i^{\top} x)),
%	\label{eq:lr}
%\end{equation}
%with the additional requirement that $b_i \in \{-1,1\}$ for all $i$.
\alert{For \cref{eq:kyfan}, we have the following parameter estimates:
$L_f = \|A^{\top} A\|_2, L_h = \lambda_1, \rho_f = \rho_h = \rho_p =
0$, and $\sigma_f$ equals to the smallest eigenvalue of $A^{\top} A$.}

Given the existence of the $p$ term in this experiment, PG is not
applicable, but by again treating $(f+h)$ as the smooth term, we can
use the proximal DC algorithm described in \cref{eq:proxdc}.
%Our experiment uses publicly available real-world data
%\ifdefined\submit 
%a9a ($m=32561$, $n=123$), cpusmall\_scale ($m=8192$, $n=12$), ijcnn1 ($m=49900$, $n=22$), phishing ($m=11055$, $n=68$) and heart ($m=270$, $n=13$).\footnote{Downloaded from \url{https://www.csie.ntu.edu.tw/~cjlin/libsvmtools/datasets/}.}\jhsolved{How about we just remove \texttt{heart}?}\cpsolved{I am fine either way.}\jhsolved{I think it's okay now. I only thought of deleting it since the table was too big previously.}
%\else
% listed in
%\cref{tbl:data}.\footnote{Downloaded from \url{https://www.csie.ntu.edu.tw/~cjlin/libsvmtools/datasets/}.}
%\fi 
In this experiment, we  set $\lambda_1 =
5$, $\lambda_2 = 10$, and $k = \lfloor n/10 \rfloor$.
As \cref{eq:delta2} does not hold for any $\tau \geq
2$, for our method we use the same range of $\tau$ as the previous experiment.
All algorithms are run with a cap of $100,000$ iterations.
Results on publicly available real-world data sets are shown in
\cref{tbl:lskyfan}.\footnote{Downloaded from
	\url{https://www.csie.ntu.edu.tw/~cjlin/libsvmtools/datasets/}.}
% \cp{I am leaving it here right now for your reference, but it still
% 	seems to me that this part is not very meaningful to be
% presented.}

In terms of the number of iterations, we can see that our method with
$\tau > 1$ is more efficient than PDC, and before reaching $\tau = 2$,
the number of iterations of
our method is roughly proportional to $1/ \tau$.
For those datasets on which $\tau = 2$ is feasible, we see that the
per-iteration decrease for $\tau = 2$ calculated in the proof of
\cref{thm:Vk_nonincreasing_>2} is very close to $0$, which already
suggested a slow convergence.
We therefore see that the value $c(\alpha)$ could be a good indicator
for the speed of practical convergence as well.
%\cp{Should we actually
%calculate this in all cases as a reference?}
%\jh{In my view, this might be insightful for us but we probably don't need to report it in the manuscript.}
%\cp{Sure.}

On the other hand, in terms of the total running time, we observe that
by utilizing the problem structure wisely, our method also achieves
significantly lower time cost per iteration.
This shows that although it is always possible to solve the problem
\cref{eq:problem} by existing methods by considering less problem
structures (in the worst case, one can always just consider it as a
nonsmooth optimization problem and use a subgradient method to solve
it), wisely exploiting problem structures can lead to much higher
practical efficiency.

\begin{table}[tbh]
	\centering
 	\caption{Comparison between proximal DC (denoted by PDC) and our method for
	\cref{eq:kyfan}. $*$ denotes that the maximum iteration of $100000$ is reached,
 $\tau = t$ denotes
our method with the specified $\tau$, and $-$ denotes that this value
of $\tau$ is not applicable for this problem.}

 \ifdefined\submit
 \small
 \fi
	\begin{tabular}{@{}l|rr@{\hspace{2pt}}|rr@{\hspace{2pt}}|rr@{\hspace{2pt}}|rr@{\hspace{2pt}}|rr@{\hspace{2pt}}}
& \multicolumn{2}{c|}{a9a}& \multicolumn{2}{c|}{cpusmall\_scale}& \multicolumn{2}{c|}{ijcnn1}& \multicolumn{2}{c|}{phishing}& \multicolumn{2}{c}{heart}\\
& $m =$ & $n = $
& $m =$ & $n = $
& $m =$ & $n = $
& $m =$ & $n = $
& $m =$ & $n = $\\
& 32561 & 123 & 8192 & 12 & 49900 & 22 & 11055 & 68 & 270 & 13\\
\cmidrule{2-11}
 Method
 & iter & time& iter & time& iter & time& iter & time& iter & time\\
 \midrule
 PDC	 & 91859 & 67.21 & 37132 & 5.64 & 2942 & 1.92 & 7283 & 3.27 & \conditionalstar{*100000} & \conditionalstar{*2.42} \\
% EPDCA	 & 1868 & 1.30 & 356 & 0.06 & 65 & 0.06 & 258 & 0.12 & 453 & 0.01 \\
% $\tau = 0.5$ & \conditionalstar{*100000} & \conditionalstar{*7.49} & 74267 & 2.54 & 5898 & 0.32 & 14566 & 0.78 & \conditionalstar{*100000} & \conditionalstar{*1.83} \\
% $\tau = 0.6$ & \conditionalstar{*100000} & \conditionalstar{*7.50} & 61889 & 2.13 & 4915 & 0.25 & 12138 & 0.65 & \conditionalstar{*100000} & \conditionalstar{*1.73} \\
% $\tau = 0.7$ & \conditionalstar{*100000} & \conditionalstar{*7.53} & 53048 & 1.89 & 4212 & 0.22 & 10403 & 0.58 & \conditionalstar{*100000} & \conditionalstar{*1.78} \\
% $\tau = 0.8$ & \conditionalstar{*100000} & \conditionalstar{*7.27} & 46416 & 1.69 & 3685 & 0.21 & 9103 & 0.50 & \conditionalstar{*100000} & \conditionalstar{*1.82} \\
% $\tau = 0.9$ & \conditionalstar{*100000} & \conditionalstar{*7.41} & 41259 & 1.52 & 3276 & 0.19 & 8091 & 0.44 & \conditionalstar{*100000} & \conditionalstar{*1.81} \\
$\tau = 1$ & 91865 & 6.80 & 37133 & 1.34 & 2948 & 0.16 & 7282 & 0.40 & \conditionalstar{*100000} & \conditionalstar{*1.77} \\
$\tau = 1.1$ & 83514 & 6.02 & 33757 & 1.23 & 2680 & 0.14 & 6620 & 0.36 & 91675 & 1.56 \\
$\tau = 1.2$ & 76555 & 5.74 & 30944 & 1.18 & 2456 & 0.14 & 6069 & 0.34 & 84035 & 1.48 \\
$\tau = 1.3$ & 70664 & 5.05 & 28564 & 1.03 & 2267 & 0.13 & 5602 & 0.32 & 77571 & 1.42 \\
$\tau = 1.4$ & 65618 & 4.85 & 26524 & 1.10 & 2105 & 0.13 & 5202 & 0.30 & 72030 & 1.36 \\
$\tau = 1.5$ & 61242 & 4.64 & 24756 & 0.97 & 1965 & 0.12 & 4856 & 0.28 & 67228 & 1.20 \\
$\tau = 1.6$ & 57415 & 4.13 & 23209 & 0.86 & 1842 & 0.11 & 4553 & 0.26 & 63027 & 1.13 \\
$\tau = 1.7$ & 54038 & 4.09 & 21844 & 0.76 & 1733 & 0.11 & 4285 & 0.24 & 59319 & 1.17 \\
$\tau = 1.8$ & 51037 & 3.88 & 20630 & 0.74 & 1637 & 0.10 & 4047 & 0.23 & 56024 & 1.01 \\
$\tau = 1.9$ & \bftab 48350 & \bftab 3.63 & \bftab 19545 & \bftab 0.68
& \bftab 1550 & \bftab 0.09 & \bftab 3835 & \bftab 0.22 & \bftab 53075
& \bftab 0.94 \\
$\tau = 2$ & - & - & - & - & 1746 & 0.11 & - & - & 97789 & 1.83 \\
\end{tabular}
	\label{tbl:lskyfan}
\end{table}

\section{Conclusion and Future Work}
\label{sec:conclusion}
In this work, we presented a new splitting algorithm
for the four-term optimization problem \eqref{eq:problem}.
Our algorithm generalizes the Davis-Yin splitting algorithm for
three-term optimization to allow for an additional nonsmooth term. We
derived stepsize estimates for the proposed algorithm that ensure
global subsequential convergence to stationary points of the objective
function. A notable implication of our results is the significant
improvement in the upper bound estimates for the stepsize of the DYS
algorithm that guarantee subsequential convergence.

\alert{Our numerical experiments demonstrate that larger stepsizes can
	significantly accelerate the convergence of our proposed
	algorithm, emphasizing the critical role of stepsize selection in
	practical performance. Previous works, such as \cite{WriNF08a,LuZ19a}, have
	shown that incorporating \cpmodifyok{adaptive stepsizes safeguarded
	by} line search techniques can improve the efficiency of PG and
	PDC, which are special cases of our framework.
	Our preliminary experiments suggest that such techniques, while not preserving iteration complexity guarantees, can substantially enhance empirical convergence speed.
 %    \cp{I don't think
	% 	this part is really that necessary in the conclusion part.
	% 	Here is my alternative version:
	% 	Our preliminary experiments show that albeit invalidating our
	% 	iteration complexity guarantees,
	% 	incorporating the practical acceleration techniques of
	% 	\cite{WriNF08a,LuZ19a} can indeed sometimes significantly
	% improve the empirical convergence speed of our algorithm.}
 %    \jh{Thanks for your suggestion. With your suggested phrasing, it sounds to me though that it implies that we did try to implement their acceleration strategies for our own four-operator splitting algorithm. Is that correct? But my understanding is that we are yet to try these strategies for our algorithm. If we don't want to directly mention that enhanced PDC outperforms us, maybe we can say ``Our preliminary experiments suggest that such techniques, while not preserving iteration complexity guarantees, can substantially enhance empirical convergence speed.''}
	% \cp{Originally my intention was that given that PDC and PG are
	% 	also special cases of the framework, EPDCA/SpaRSA technically
	% 	can also be considered as a kind of acceleration of the
	% 	framework, but indeed it's too ambiguous and misleading.
	% Your version looks good.}
% However, we also observe that the method in \cite{LuZ19a} often converges to an objective value significantly higher than that obtained by other methods, suggesting that its faster convergence may come at the expense of suboptimal final solutions. 
Motivated by these findings, we are actively exploring extensions of
adaptive stepsize strategies to our more general algorithm, aiming to
further enhance efficiency while maintaining solution quality \cpmodifyok{and the
convergence rate guarantees}.}
% \footnote{AT: Thank you for the improvement on this part! I think Harold's last suggestion sounds good. }

% \ifdefined\submit
% \section*{Declarations}
% % CP's research was supported in part by KAKENHI of the 
%  %NSTC of R.O.C. grant 109-2222-E001-003. 
% CP's research is supported in part by the JSPS Grant-in-Aid for Research
% Activity Start-up 23K19981 and Grant-in-Aid for Early-Career
% Scientists 24K20845.

% \section*{Compliance with Ethical Standards}
% \else 
% \fi 

\ifdefined\submit 
\bibliographystyle{siamplain}
\else
\fi 
\bibliography{bibfile}
\end{document}